\documentclass[11pt]{amsart}
\usepackage{mabliautoref}
\usepackage{amssymb,amsthm,amsmath}
\RequirePackage[dvipsnames,usenames]{xcolor}
\usepackage{hyperref}
\usepackage{mathtools}
\usepackage[all]{xy}
\usepackage{tikz}
\usepackage{tensor}
\usepackage{enumitem}
\usepackage{chngcntr}
\usepackage{stmaryrd}

\hypersetup{
bookmarks,
bookmarksdepth=3,
bookmarksopen,
bookmarksnumbered,
pdfstartview=FitH,
colorlinks,backref,hyperindex,
linkcolor=Sepia,
anchorcolor=BurntOrange,
citecolor=MidnightBlue,
citecolor=OliveGreen,
filecolor=BlueViolet,
menucolor=Yellow,
urlcolor=OliveGreen
}

\DeclareMathAlphabet{\mathchanc}{OT1}{pzc}%
                                 {m}{it}

\newcommand{\mcH}{\mathchanc{H}}

\newcommand{\mcm}{\mathchanc{m}}

\newcommand{\mco}{\mathchanc{o}}

\newcommand{\bC}{\mathbb{C}}

\newcommand{\bF}{\mathbb{F}}

\newcommand{\bP}{\mathbb{P}}
\newcommand{\bQ}{\mathbb{Q}}

\newcommand{\bZ}{\mathbb{Z}}

\newcommand{\scr}{\mathcal}
\newcommand{\sA}{\scr{A}}
\newcommand{\sB}{\scr{B}}

\newcommand{\sD}{\scr{D}}
\newcommand{\sE}{\scr{E}}
\newcommand{\sF}{\scr{F}}
\newcommand{\sG}{\scr{G}}
\newcommand{\sH}{\scr{H}}
\newcommand{\sI}{\scr{I}}

\newcommand{\sK}{\scr{K}}
\newcommand{\sL}{\scr{L}}
\newcommand{\sM}{\scr{M}}
\newcommand{\sN}{\scr{N}}
\newcommand{\sO}{\scr{O}}

\newcommand{\sQ}{\scr{Q}}

\newcommand{\os}{\overline{s}}

\newcommand{\oJ}{\overline{J}}

\DeclareMathOperator{\Tors}{{Tors}}

\DeclareMathOperator{\vol}{{vol}}

\DeclareMathOperator{\Tr}{Tr}
\DeclareMathOperator{\codim}{codim}

\DeclareMathOperator{\Exc}{Exc}

\newcommand{\sHom}[0]{{\mcH\mco\mcm}}

\DeclareMathOperator{\Id}{{Id}}
\DeclareMathOperator{\im}{{im}}

\DeclareMathOperator{\Pic}{Pic}

\DeclareMathOperator{\reg}{reg}

\DeclareMathOperator{\rk}{{rk}}

\DeclareMathOperator{\Spec}{{Spec}}
\DeclareMathOperator{\coeff}{{coeff}}

\newcommand{\factor}[2]{\left. \raise 2pt\hbox{\ensuremath{#1}} \right/
        \hskip -2pt\raise -2pt\hbox{\ensuremath{#2}}}

\counterwithin*{equation}{section}
\counterwithin*{equation}{subsection}

\makeatletter
\renewcommand\subsection{
  \renewcommand{\sfdefault}{pag}
  \@startsection{subsection}%
  {2}{0pt}{.8\baselineskip}{.4\baselineskip}{\raggedright
    \sffamily\itshape\small\bfseries
  }}
\renewcommand\section{
  \renewcommand{\sfdefault}{phv}
  \@startsection{section} %
  {1}{0pt}{\baselineskip}{.8\baselineskip}{\centering
    \sffamily
    \scshape
    \bfseries
}}
\makeatother

\renewcommand{\theenumi}{\arabic{enumi}}

\usepackage[left=1.02in,top=1.0in,right=1.02in,bottom=1.0in]{geometry}

\title{Frobenius techniques in birational geometry}
\author{Zsolt Patakfalvi}

\begin{document}

\maketitle

\tableofcontents

\begin{abstract}
This is a survey for the 2015 AMS Summer Institute on Algebraic Geometry about the Frobenius type techniques recently used extensively in positive characteristic algebraic geometry.  We first explain the basic ideas through simple versions of the fundamental definitions and statements, and then we survey  most of the recent algebraic geometry results obtained using these techniques.
\end{abstract}

\section{Introduction}

Let $A$ be a ring over a field $k$ of characteristic $p>0$. 
The \emph{absolute Frobenius homomorphism}  $F_A$ of $A$ is the homomorphism $A \to A$ defined by $F(x):=x^p$. It is easy to see that this is indeed  a ring-homomorphism  using that in the binomial expansion of $(x+y)^p$ all the mixed terms are divisible by $p$.  Furthermore, 
\begin{enumerate}
 \item $F_A$ is a functorial homomorphism, that is, it commutes with any  homomorphisms of rings over $k$, and
\item $F_A$ induces identity on $\Spec A$, since for any prime ideal $q \subseteq A$, $x^p \in q \Rightarrow x \in q$ by the prime property. 
\end{enumerate}
In particular, for any scheme $X$ over $k$, one obtain the \emph{absolute Frobenius homomorphism} $F_X : X \to X$, which is identity on the underlying topological space $|X|$, and for every open set $U \subseteq X$, $F_X^{\#}(U)$ is the absolute Frobenius homomorphism of the ring $\sO_X(U)$. 

This survey is about techniques that become very fertile in the past 5 years in algebraic geometry over $k$, using structures arising from absolute Frobenius morphisms. The basic idea behind these techniques is the following: given a scheme $X$ over $k$, the absolute Frobenius morphism of $X$ comes with a structure homomorphism $\sO_X \to F_* \sO_X$. In particular, this homomorphism endows $F_* \sO_X$ with a $\sO_X$-module structure, which is coherent in the most algebro-geometric situations, e.g., if $X$ is quasi-projective over $k$ and $k$ is perfect.  Then the investigation of the following deliberately vague question led eventually to techniques discussed here:

\begin{question}
\label{qtn:module}
What is the $\sO_X$-module structure of $F_* \sO_X$?
\end{question}

There are two  ways one can branch from \autoref{qtn:module}. One can ask for the global structure, when $X$ is projective, or one can ask for a local module structure, when $X$ is affine. Historically the latter appeared first, in commutative algebra. One of the first milestones in this study was the proof in 1969 by Kunz of the statement that locally $F_* \sO_X$ is a free $\sO_X$-module if and only if $X$ is regular \cite{Kunz_Characterizations_of_regular_local_rings_for_characteristic_p}. This foreshadowed deep connections to singularity theory that materialized a few decades later. Then further, investigations of the module structure followed, eventually growing into a major area of  commutative algebra (e.g., \cite{Hochster_Roberts_The_purity_of_the_Frobenius_and_local_cohomology}, \cite{Hochster_Huneke_Tight_closure_invariant_theory_and_the_Brianc_con-Skoda_theorem}, etc). The investigation of the global module structure originates from 
 representation theory (e.g.,
\cite{Mehta_Ramanathan_Frobenius_splitting_and_cohomology_vanishing_for_Schubert_varieties,
Ramanan_Ramanathan_Projective_normality_of_flag_varieties_and_Schubert_varieties}). In the early 90's the connection between the two perspectives was realized and then they were connected to algebraic geometry 
(e.g., \cite{Mehta_Srinivas_A_characterization_of_rational_singularities,Hara_A_characterization_of_rational_singularities,Smith_The_multiplier_ideal_is_a_universal_test_ideal}). The typical statements of this period were connecting  module theoretic notions to singularity theoretic notions of algebraic geometry. Purely geometric applications, that is, statements that had nothing to do with module theory, and for which the module theoretic input was used only in the proofs, came about only much more recently. Primarily by the presentation and generalization of the commutative algebra arguments to a more geometric language by Karl Schwede (e.g., \cite{Schwede_F_adjunction,Schwede_A_canonical_linear_system}), the first part of the present decade brought many results in positive characteristic geometry, including topics such as Minimal Model Theory 
\cite{Hacon_Xu_On_the_three_dimensional_minimal_model_program_in_positive_characteristic,Birkar_Existence_of_flips_and_minimal_models_for_3_folds_in_char_p,
Xu_On_base_point_free_theorem_of_threefolds_in_positive_characteristic,Birkar_Waldron_Existence_of_Mori_fibre_spaces_for_3_folds_in_char_p,Waldron_Finite_generation_of_the_log_canonical_ring_for_3-folds_in_char_p,Waldron_The_LMMP_for_log_canonical_3-folds_in_char_p,Cascini_Tanaka_Xu_On_base_point_freeness_in_positive_characteristic}, semi-positivity and subaddivitiy of Kodaira dimension 
\cite{Patakfalvi_Semi_positivity_in_positive_characteristics,Patakfalvi_On_subadditivity_of_Kodaira_dimension_in_positive_characteristic_over_a_general_type_base,Chen_Zhang_The_subadditivity_of_the_Kodaira-dimension_for_fibrations_of_relative_dimension_one_in_positive_characteristics,Birkar_Chen_Zhang_Iitaka_s_C_n_m_conjecture_for_3-folds_over_finite_fields,Ejiri_Weak_positivity_theorem_and_Frobenius_stable_canonical_rings_of_geometric_generic_fibers,Zhang_Subadditivity_of_Kodaira_dimensions_for_fibrations_of_three-folds_in___positive_characteristics,Ejiri_Iitaka_s_C_n_m_conjecture_for_3-folds_in_positive_characteristic}, Seshadri constants 
\cite{Mustata_Schwede_A_Frobenius_variant_of_Seshadri_constants}, numerical dimension 
\cite{Mustata_The_non_nef_locus_in_positive_characteristic,Cascini_McKernan_Mustata_The_augmented_base_locus_in_positive_characteristic,Cascini_Hacon_Mustata_Schwede_On_the_numerical_dimension_of_pseudo_effective_divisors_in_positive_characteristic}, rationally connectedness and the geometry of Fano varieties \cite{Gongyo_Li_Patakfalvi_Schwede_Tanaka_Zong_On_rational_connectedness_of_globally_F_regular_threefolds,Gongyo_Nakamura_Tanaka_Rational_points_on_log_Fano_threefolds_over_a_finite_field,Wang_On_relative_rational_chain_connectedness_of_threefolds_with_anti-big___canonical_divisors_in_positive_characteristics,Ejiri_Positivity_of_anti-canonical_divisors_and_F-purity_of_fibers} generic vanishing and other topics about abelian varieties (e.g. singularities of the Theta divisor) and varieties of maximal Albanese dimension
\cite{Hacon_Singularities_of_pluri_theta_divisors_in_Char_p,Zhang_Pluri-canonical_maps_of_varieties_of_maximal_Albanese_dimension_in___positive_characteristic,Hacon_Patakfalvi_A_generic_vanishing_in_positive_characteristic,Watson_Zhang_On_the_generic_vanishing_theorem_of_Cartier_modules,Hacon_Patakfalvi_On_charactarization_of_abelian_varieties_in_characteristic_p,Sannai_Tanaka_A_characterization_of_ordinary_abelian_varieties_by_the_Frobenius___push-forward_of_the_structure_sheaf,Wang_Generic_vanishing_and_classification_of_irregular_surfaces_in_positive___characteristics},  Kodaira type vanishings, surjectivity statements and liftability to characteristic $0$
\cite{Tanaka_The_X-method_for_klt_surfaces_in_positive_characteristic,Tanaka_The_trace_map_of_Frobenius_and_extending_sections_for_threefolds,Cascini_Tanaka_Witaszek_On_log_del_Pezzo_surfaces_in_large_characteristic,Zdanowicz_Liftability_of_singularities_and_their_Frobenius_morphism_modulo_p_2} canonical bundle formula \cite{Das_Schwede_The_F-different_and_a_canonical_bundle_formula,Das_Hacon_On_the_Adjunction_Formula_for_3-folds_in_characteristic_p>5}, inversion of adjunction \cite{Das_On_strongly_F_regular_inversion_of_adjunction}. We refer to  \autoref{sec:application} for more detailed list and explanation on these results. 

\subsection{Structure} In \autoref{sec:basic_notions} we present the fundamental definitions of the area as well as some of the classical statements, where classical means that they were (mostly) proven before 2000. In \autoref{sec:newer_methods} we present more recent statements that are geared more towards birational geometry applications. In particular, all of them are centered around finding sections of line bundles in the presence of some positivity. Lastly, in \autoref{sec:application} we survey the recent applications to higher dimensional algebraic geometry of the methods presented in the previous sections. 

\subsection{Acknowledgement}

The author of the article was supported by the NSF grant DMS-1502236.

\section{Setup}

For simplicity, throughout the article,  we are working on either quasi-projective varieties $X$ over a perfect field $k$ of characteristic $p>0$, or on $\Spec \sO_{X,x}$, where $X$ is as above, and $x \in X$ is arbitrary. 

\section{Basic notions - fundamental results}
\label{sec:basic_notions}

\subsection{$F$-purity and global $F$-purity}
\label{sec:F_pure}

One of the simplest special case of \autoref{qtn:module} is the following:

\begin{question}
\label{qtn:splitting}
Does the map $\sO_X \to F_* \sO_X$ split as a homomorphism of $\sO_X$-modules?
\end{question}

In fact, if the above splitting occurs locally, we call $X$ \emph{$F$-pure}, and if it happens globally we call it \emph{globally $F$-pure}. In this section we explain some of the relevance of these two notions to algebraic geometry. We start with (local) $F$-purity.

\begin{remark}
The splitting asked for in \autoref{qtn:splitting}, is equivalent to the natural evaluation map $\sHom_X(F_* \sO_X, \sO_X) \to \sO_X$ given by $\psi \mapsto \psi(1)$ being surjective. Then one can deduce that $F$-purity is a local notion. That is, $X$ is $F$-pure if and only if $\Spec \sO_{X,x}$ if $F$-pure for all $x \in X$ \cite[Exc 2.9]{Schwede_Tucker_A_survey_of_test_ideals}.  In particular, $X$ is $F$-pure if and only if the splitting asked in \autoref{qtn:splitting} happens at every local ring, or equivalently on any affine cover. 

\end{remark}

  If local equations are known then verifying $F$-purity is algorithmic. To state the precise statement, let us recall that for an ideal $I \subseteq R$ in a ring over $k$, $I^{[p]}:=\left(f_1^p, \dots, f_r^p \right)$, where $f_1, \dots, f_r$ is an arbitrary set of generators of $I$ (one can show that this is independent of the choice of generators, e.g., \cite[Exercise 2.12]{Schwede_Tucker_A_survey_of_test_ideals}).

\begin{theorem}
\cite[Lemma 1.6]{Fedder_F_purity_and_rational_singularity}
Let $X:=\Spec S$, where $S:= k[x_1,\dots,x_n]/I$, and let $\oJ$ be a prime ideal of $S$, and $J$ the preimage of $\oJ$ in $k[x_1,\dots,x_n]$. Then $S_{\oJ}$ is $F$-pure if and only if $\left(I^{[p]}:I \right) \not\subseteq J^{[p]}$.
\end{theorem}

\begin{corollary}
\label{cor:Fedder}
A hypersurface singularity $\Spec \left( k[x_1,\dots,x_n]/(f) \right)$ is $F$-pure at the origin if and only if $f^{p-1} \not\in \left(x_1^p,\dots,x_n^p \right)$.
\end{corollary}

\begin{example}
\label{ex:F_pure}
The cone over the Fermat cubic $\{x^3+y^3 + z^3=0\}$ is $F$-pure if and only if $p \equiv 1 \mod 3$. This in turn is equivalent to the corresponding elliptic curve being ordinary.

One can show easily the above statement using \autoref{cor:Fedder}. For example, if $p \equiv 1 \mod 3$, then $\left(x^3+y^3 + z^3\right)^{p-1}$ contains the non-zero monomial $\frac{(p-1)!}{\frac{p-1}{3}!\frac{p-1}{3}!\frac{p-1}{3}!}x^{p-1}y^{p-1}z^{p-1}$, which is not in $\left( x^p,y^p,z^p \right)$, and all the other monomials of $\left(x^3+y^3 + z^3\right)^{p-1}$ are in $\left(x^p,y^p,z^p \right)$. In particular, $\left(x^3+y^3+z^3\right)^{p-1} \not\in \left(x^p,y^p,z^p\right)$. 
\end{example}

As cones over elliptic curves are the typical examples of log canonical surface singularities, \autoref{ex:F_pure} might indicate already that $F$-purity is closely connected to the usual notions of the classification theory of algebraic varieties. In fact,  $F$-pure singularities are connected to log canonical singularities, and projective, globally $F$-pure varieties are connected to Calabi-Yau type varieties. 

Before stating any precise statement about this connection, let us recall first these notions. Let $(X, \Delta)$ be a pair, that is, $X$ is normal, and $\Delta$ is an effective $\bQ$-divisor, i.e., a formal sum $\sum a_i D_i$, where $a_i \in \bQ$ and $D_i$ are irreducible Weil divisors. Then $(X,\Delta)$ has log canonical singularities if $K_X + \Delta$ is $\bQ$-Cartier, and for each normal variety $Y$, and birational proper map $f : Y \to X$,  all the coefficient of $\Gamma$ are at most $1$, where $\Gamma$ is the unique $\bQ$-divisor for which 
\begin{equation}
\label{eq:lc_def}
K_Y + \Gamma = f^* (K_X + \Delta)
\end{equation}
holds. Note that in \autoref{eq:lc_def} we require actual equation, as opposed to linear equivalence, and we assume that $K_X = f_* K_Y$. If $X$ is just a normal variety, without any further assumption, we say that $X$ has \emph{log canonical singularities} if and only if there is an effective $\bQ$-divisor $\Delta$ on $X$ such that $(X, \Delta)$ is log canonical. Similarly a projective variety $X$ is of \emph{Calabi-Yau type} if there is an effective $\bQ$-divisor $\Delta$ on $X$, such that $(X, \Delta)$ is log canonical and $K_X + \Delta \sim_{\bQ} 0$ ($\bQ$-linear equivalence means that some multiples are linearly equivalent).

One direction of the above mentioned connection is  in fact not hard to prove, and we will show it in \autoref{sec:F_singularities}:

\begin{theorem} \cite[Thm 3.3]{Hara_Watanabe_F_regular_and_F_pure_rings_vs_log_terminal_and_log_canonical} \cite[Thm 4.3 \& Thm 4.4]{Schwede_Smith_Globally_F_regular_an_log_Fano_varieties}
\label{thm:F_pure_lc_CY}
Let $X$ be normal. 
\begin{enumerate}
 \item If $X$ is $F$-pure, then it is log canonical, that is, there is an effective $\bQ$-divisor $\Delta$, such that $(X, \Delta)$ is log canonical.
\item If $X$ is projective and globally $F$-pure, then it is of Calabi-Yau type, that is, there is an effective $\bQ$-divisor $\Delta$, such that $(X, \Delta)$ is log canonical and $K_X + \Delta \sim_{\bQ} 0$.
\end{enumerate}
\end{theorem}

In the case of reductions of characteristic $0$ varieties mod $p$, one also has a backwards statement, conditional on the following arithmetic conjecture:

\begin{conjecture}
\label{conj:weak_ordinarity} { ( \scshape Weak ordinarity conjecture ) } 

Let $Y$ be a smooth, connected projective variety over an algebraically closed field $k'$ of characteristic $0$. Given a model $Y_A$ of $Y$ over a finitely generated $\bZ$-algebra $A$, the set
\begin{equation*}
\{ s\in \Spec A | \textrm{$s$ is a closed point such that the action of Frobenius on $H^{\dim Y_s}\left(Y_s, \sO_{Y_s}\right)$ is bijective} \}
\end{equation*}
is dense in $\Spec A$. 
\end{conjecture}

\begin{theorem} \cite{Mustata_Srinivas_Ordinary_varieties_and,Takagi_Adjoint_ideals_and_a_correspondence_between_log_canonicity_and_F-purity,Bhatt_Schwede_Takagi_The_weak_ordinarity_conjecture_and_F-singularities}
\label{thm:F_pure_reduction}
Let $X$ be a log canonical singularity over an algebraically closed field $k'$ of characteristic $0$. Given a model $X_A$ of $X$ over a finitely generated $\bZ$-algebra $A$, 
\begin{equation*}
\{ s\in \Spec A | \textrm{$s$ is a closed point such that $X_s$ is $F$-pure} \}
\end{equation*}
is dense in $\Spec A$, if we assume \autoref{conj:weak_ordinarity}. 
\end{theorem}

The proof of \autoref{thm:F_pure_reduction} is beyond the scope of the present paper, so we refer to the above references.

\begin{remark}
In some special cases there is a backwards implication as in \autoref{thm:F_pure_reduction} also in equicharacteristic. In particular, in dimension $2$, most log canonical singularities are $F$-pure \cite{Hara_Classification_of_two_dimensional_F_regular_and_F_pure_singularities}. For an example statement with proof along this line see \autoref{cor:A_1}.
\end{remark}

Assume we have a (projective) Calabi-Yau type  variety $X$ in characteristic zero. That is,  $X$ is projective, and there exists an effective $\bQ$-divisor $\Delta$ on $X$, such that $(X,\Delta)$ is log canonical and $K_X + \Delta \sim_{\bQ} 0$. Let $L$ be an ample Cartier divisor on $X$. Then we may write $L \sim_{\bQ} L + K_X + \Delta$. Hence, the log canonical version of Kodaira vanishing \cite[Thm 2.42]{Fujino_Introduction_to_the_log_minimal_model_program_for_log_canonical_pairs} yields that $H^i(X, L)=0$. This is  a fundamental property of Calabi-Yau type varieties, which is to a large extent responsible for the Mori dream space property of Fano type varieties (special cases of Calabi-Yau type varieties defined in \autoref{sec:F_regular}). 

Surprisingly, the above vanishing holds in our situation, so in characteristic $p>0$, for globally $F$-split varieties:

\begin{theorem}
\label{thm:Kodaira_vanishing_F_pure}
Let $X$ be a normal, globally $F$-split  projective variety, and let $L$ be an ample Cartier divisor on $X$. Then the following holds.
\begin{enumerate}
 \item \label{itm:Kodaira_vanishing_F_regular:positive} $H^i(X, \sO_X(L))=0$ for $i>0$. 
\item \label{itm:Kodaira_vanishing_F_regular:negative} if $X$ is Cohen-Macaulay, then $H^i(X, \sO_X(-L))=0$ for $i< \dim X$. 
\end{enumerate}
 
\end{theorem}

\begin{proof}
Let $\psi: F_* \sO_X \to \sO_X$ be the splitting of $\sO_X \hookrightarrow F_* \sO_X$ guaranteed by the global $F$-purity of $X$. The map $\psi$ being a splitting is equivalent to the condition $\phi(1)=1$. Hence, for any integer $e >0$, $\psi \circ F_* (\psi) \circ \dots \circ F^{e-2}_*(\psi) \circ F^{e-1}_* (\psi) $ yields a splitting of the natural morphism $\sO_X \hookrightarrow F^e_* \sO_X$, where $F^j$ denoted the $j$ times composition $\underbrace{F \circ \dots \circ F}_{j \textrm{ times}}$. That is, we have a commutative diagram such as:
\begin{equation*}
\xymatrix{
\sO_X \ar[r] \ar@/^1.5pc/[rr]^{\Id} & F_*^e \sO_X \ar[r] & \sO_X
}
\end{equation*}
By first tensoring the above diagram by $\sO_X(L)$ and then applying $H^i(\_)$ to it, plus using that
\begin{itemize}
 \item $\sO_X(L) \otimes F_*^e \sO_X \cong  F_*^e \left(F^e\right)^* \sO_X(L)$ by the projection formula,
\item $\left(F^e\right)^* \sO_X(L) \cong \sO_X(p^eL)$, since $\left(F^e\right)^*$ raises each line bundle to the $p^e$-th power (because it raises the gluing functions to the $p$-th power),
\item  $H^i(X, F_*^e \sO_X(p^eL)) \cong H^i(X, \sO_X(p^eL))$, since $F^e$ is an affine morphism, 
\end{itemize}
we obtain the commutative diagram:
\begin{equation*}
\xymatrix{
H^i(X, \sO_X(L)) \ar[r] \ar@/^1.5pc/[rr]^{\Id} &  H^i(X,\sO_X(L) \otimes F_*^e \sO_X)   \cong H^i(X, \sO_X(p^eL)) \ar[r] & H^i(X,\sO_X(L))
}
\end{equation*}
Next, we note that by Serre vanishing $H^i(\sO_X(p^eL))=0$ for $e \gg 0$, which then implies that $H^i(X, \sO_X(L))=0$, which is exactly the statement of \autoref{itm:Kodaira_vanishing_F_regular:positive}. The proof of point \autoref{itm:Kodaira_vanishing_F_regular:negative} proceeds along the same line, except one needs the vanishing of $H^i(\sO_X(-p^eL))=0$ for $e \gg 0$, which follows from Serre duality, when $X$ is Cohen-Macaulay \cite[Thm III.7.6]{Hartshorne_Algebraic_geometry}.
\end{proof}


\autoref{thm:Kodaira_vanishing_F_pure} leads us to the first purely algebro geometric application to the theory: if global $F$-purity is  known for a class of varieties, then Kodaira vanishing is known for this class. For this, one has to find such a class:

\begin{example} \cite[Example (3.6)]{Hara_A_characterization_of_rational_singularities_in_terms_of_injectivity_of_Frobenius_maps}
If $X$ is a smooth del Pezzo surface over $k$, and $p >5$, then $X$ is globally $F$-pure. In fact, in \cite[Example (3.6)]{Hara_A_characterization_of_rational_singularities_in_terms_of_injectivity_of_Frobenius_maps} it is stated that such $X$ is globally $F$-regular, which implies global $F$-purity as we will see in \autoref{sec:F_regular}.
\end{example}

Then one obtains the following (which is well known in fact for the more general class of surfaces of special type except the quasi-elliptic surfaces of Kodaira dimension $1$ \cite[Thm 1.6]{Ekedahl_Canonical_models_of_surfaces_of_general_type_in_positive_characteristic}, so this is admittedly a quite weak example application only):

\begin{corollary}
If $X$ is a smooth del Pezzo surface, and $p >5$, then Kodaira vanishing holds on $X$. 
\end{corollary}

\subsection{$F$-regularity}
\label{sec:F_regular}

Here we discuss the different notions of $F$-regularity that relate to Kawamata log terminal singularities and Fano type varieties as $F$-purity is relating to log canonical singularities and Calabi-Yau type varieties. In fact, as we will see it soon, the statements are even nicer in this case, as $F$-regularity is  the version of $F$-purity that is stable under perturbations. This again parallels the log canonical and Kawamata log canonical analogy, since the latter is the version of the formal that is stable under perturbations by any effective divisor.

\begin{definition} \cite{Hochster_Huneke_Tight_closure_and_strong_F-regularity,Smith_Globally_F-regular_varieties_applications_to_vanishing_theorems_for_quotients_of_Fano_varieties,Schwede_Smith_Globally_F_regular_an_log_Fano_varieties}
\label{def:F_regular}
Let $X$ be affine and normal. Then  $X$ is said to be \emph{strongly $F$-regular} if for all effective divisors $D \geq 0$, the composition of the following natural maps splits for some integer $e > 0$:
\begin{equation*}
\sO_X \hookrightarrow F^e_* \sO_X \hookrightarrow F^e_* (\sO_X(D) ). 
\end{equation*}
(Here $F^e$ denotes the $e$-times composition of $F$ with itself.)

Similarly, if $X$ is projective and normal, then $X$ is said to be \emph{globally $F$-regular} if for all effective divisors $D \geq 0$, the composition of the following natural maps splits for some integer $e > 0$:
\begin{equation*}
\sO_X \hookrightarrow F^e_* \sO_X \hookrightarrow F^e_* (\sO_X(D) ). 
\end{equation*}
\end{definition}

\begin{remark}
A few remarks on \autoref{def:F_regular}:
\begin{enumerate}
\item The two definitions in \autoref{def:F_regular} are indeed formally completely the same, the only difference is that in the first case we assume $X$ to be affine, while in the second one we assume it to be projective. In particular, the latter is a global property, and as we see from the following point, the former is a singularity property.
\item One can show that strong $F$-regularity is a local property, that is, $X$ is strongly $F$-regular if and only if so are all its local rings \cite[Thm 3.1.a]{Hochster_Huneke_Tight_closure_and_strong_F-regularity}. Hence, one defines in general $X$ to be strongly $F$-regular, if any of its affine covers is strongly $F$-regular, or equivalently, if all its local rings are strongly $F$-regular. 

\item By restricting the splitting of \autoref{def:F_regular} to $F_* \sO_X \subseteq F^e_* (\sO_X(D) )$ we see that a strongly $F$-regular scheme is $F$-pure and a globally $F$-regular scheme is globally $F$-pure.
\end{enumerate}

\end{remark}

Strong and global $F$-regularity are connected to the notions of Kawamata log-terminal and Fano type varieties analogously to how (global) $F$-purity is connected to the notions of log canonical and Calabi-Yau type. For this let us recall that the definition of \emph{Kawamata log terminal} singularities is verbatim the same as of log canonical singularities, except one requires from $\Gamma$, defined in \autoref{eq:lc_def}, to have coefficients less than $1$, instead of less  than or equal to $1$. Similarly, being Fano type is analogous to being Calabi-Yau type. That is, $X$ is \emph{Fano type} if there is an effective $\bQ$-divisor $\Delta$ on $X$ such that $(X,\Delta)$ is Kawamata log-terminal, and $-(K_X +\Delta)$ is ample. Then the $F$-regular version of \autoref{thm:F_pure_lc_CY}, is verbatim the same except one has to make the following replacements:

\begin{theorem} \cite[Thm 3.3]{Hara_Watanabe_F_regular_and_F_pure_rings_vs_log_terminal_and_log_canonical} \cite[Thm 4.3 \& Thm 4.4]{Schwede_Smith_Globally_F_regular_an_log_Fano_varieties}
\label{thm:F_regular_klt_Fano}
Let $X$ be normal. 
\begin{enumerate}
 \item If $X$ is strongly $F$-regular, then $X$ has Kawamata log-terminal singularities.
\item If $X$ is projective and globally $F$-regular, then it is of Fano type.
\end{enumerate}
\end{theorem}

Since the proof of \autoref{thm:F_regular_klt_Fano} is very similar to that of \autoref{thm:F_pure_lc_CY}, we omit it. Despite the statements of \autoref{thm:F_pure_lc_CY} and \autoref{thm:F_regular_klt_Fano} being verbatim the same,  in general strongly $F$-regular singularities and globally $F$-regular varieties behave better than their $F$-pure counterparts. This parallels the well-known phenomenon that in characteristic zero Kawamta log terminal singularities and Fano type varieties behave better than log canonical singularities and Calabi-Yau type varieties. For example, smooth Fano varieties are bounded \cite{Kollar_Miyaoka_Mori_Rational_connectedness_and_boundedness_of_Fano_manifolds}, while smooth (algebraic) Calabi-Yau varieties are not (although it is conjectured that over $\bC$ their topological types are bounded). The first instance of the above phenomenon is that the reduction theorem is stronger for strongly $F$-regular varieties, than \autoref{thm:F_pure_reduction}. Indeed, on does not need to 
assume the arithmetic conjecture \autoref{conj:weak_ordinarity}:

\begin{theorem} \cite{Takagi_A_characteristic_p_analogue_of_plt_sintularities,Takagi_An_interpretation_of_multiplier_ideals_via_tight_closure,Hara_Yoshida_A_generalization_of_tight_closure_and_multiplier_ideals,Hara_A_characterization_of_rational_singularities_in_terms_of_injectivity_of_Frobenius_maps,Mehta_Srinivas_A_characterization_of_rational_singularities}
\label{thm:F_regular_reduction}
Let $X$ be a Kawamata log terminal singularity (resp. a Fano type variety)  over an algebraically closed field $k'$ of characteristic $0$. Given a model $X_A$ over a finitely generated $\bZ$-algebra $A$, 
\begin{equation*}
\{ s\in \Spec A | \textrm{$s$ is a point such that $X_{\os}$ is strongly $F$-regular (resp. globally $F$-regular)} \}
\end{equation*}
is open and dense in $\Spec A$. 
\end{theorem}

%

We note that for globally $F$-regular varieties, one can also strengthen the Kodaira vanishing result of \autoref{thm:Kodaira_vanishing_F_pure} in different ways, for example by allowing nef and big divisors \cite[4.2-4.4]{Smith_Globally_F-regular_varieties_applications_to_vanishing_theorems_for_quotients_of_Fano_varieties}, \cite[Thm 6.8]{Schwede_Smith_Globally_F_regular_an_log_Fano_varieties}.


\subsection{Duality theory}
\label{sec:duality}

The vast use of the above defined notions in algebraic geometry was to a great extent due to a systematic study of the dual formulation of the above notions (e.g., \cite{Schwede_F_adjunction}). Here, in \autoref{rmk:duality}, we summarize the most important facts of duality theory needed for further investigation. These concern the most manageable  case of the theory, that is, the case of finite surjective maps, which is already discussed and proven in Hartshorne's widely used graduate text book \cite{Hartshorne_Algebraic_geometry}.

\begin{remark}
\label{rmk:duality}
First, recall the definition of $f^!$  for finite morphisms $f: X \to Y$ \cite[Exc III.6.10.]{Hartshorne_Algebraic_geometry}. Let $\sG$ be a quasi-coherent $\sO_Y$-module. Then $\sHom_{\sO_Y}(f_* \sO_X, \sG)$ has a natural $f_* \sO_X$ module structure. In particular, there is an $\sO_X$-module $f^! \sG$, defined up to isomorphism, such that   $f_* f^! \sG \cong \sHom_{\sO_Y}(f_* \sO_X, \sG)$. These obey the following properties:
\begin{enumerate}
\item \label{itm:f_upper_shriek} According to \cite[Exc III.7.2]{Hartshorne_Algebraic_geometry} $f^! \omega_{Y}^0 \cong \omega_{X}^0$, where $\omega_{Y}^0$ and $\omega_{X}^0$ are the dualizing sheaves of projective equidimensional schemes over  $k$ defined in \cite[p 241, Def]{Hartshorne_Algebraic_geometry}, and will be denoted by $\omega_Y$ and $\omega_X$ here. In fact, the above isomorphism can be taken to be unique with a correct setup of the theory \cite{Hartshorne_Residues_and_duality}, which in the language of \cite[p 241, Def]{Hartshorne_Algebraic_geometry} means remembering the trace map. This subtlety is usually indifferent for the methods presented in the present paper, hence we disregard it for simplicity.

\item \label{itm:trace} For every quasi-coherent sheaf $\sG$ on $Y$ one can define a trace morphism $\Tr_{f, \sG} : f_* f^! \sG \to \sG$ as follows: since $f_* f^! \sG \cong \sHom_{\sO_Y}(f_* \sO_X, \sG)$, $\Tr_{f, \sG}$ is identified with the natural evaluation map $ \sHom_{\sO_Y}(f_* \sO_X, \sG) \ni \phi \mapsto \phi(1) \in \sG$. In particular, this yields a trace morphism $\Tr_f : f_* \omega_X \cong f_* f^! \omega_Y \to \omega_Y$. 

\item \label{itm:duality} According to \cite[Exercise III.6.10.b]{Hartshorne_Algebraic_geometry}, there is a natural isomorphism for every finite, surjective morphism $f : X \to Y$, and quasi-coherent sheaves $\sF$ and $\sG$ on $X$ and $Y$, respectively:
\begin{equation}
\label{eq:duality_finite_map}
f_* \sHom_X(\sF, f^! \sG) \cong \sHom_Y(f_* \sF, \sG).
\end{equation}

\item \label{itm:f_upper_shriek_twist_line_bundle} If $\sL$ is a line bundle on $Y$ and $\sG$ a quasi-coherent sheaf on $Y$, then $f^! (\sG \otimes \sL) \cong (f^! \sG) \otimes f^* \sL$. Indeed,
\begin{equation*}
\sHom_Y(f_* \sO_X, \sG \otimes \sL) \cong \sHom_Y(f_* \sO_X, \sG) \otimes \sL \cong (f_* f^! \sG) \otimes \sL \cong f_* ((f^! \sG) \otimes f^* \sL),
\end{equation*}
where we used the projection formula in the last step. 

\item \label{itm:f_upper_shriek_twist_divisor} For the last property, recall that if $X$ is normal, then there is a good theory of reflexive sheaves of rank one, which is equivalent to the theory of Weil divisors modulo rational equivalence. This generalizes the usual equivalence of line bundles and Cartier divisors modulo linear equivalence. Here a sheaf $\sE$ of rank $1$ is reflexive if it is a rank $1$ coherent sheaf such that the natural map $\sE \to \sE^{**}$ to the double dual is an isomorphism. Equivalently $\sE$ is reflexive if and only if $\sE \cong \iota_* \left(\sE|_{X_{\reg}} \right)$, where $\iota : X_{\reg} \hookrightarrow X$ is the usual embedding \cite[Prop 1.6]{Hartshorne_Stable_reflexive_sheaves}. Furthermore, for an arbitrary coherent sheaf $\sF$ of rank $1$ on $\sE$, $\sF^{**}$ is called the reflexive hull, which is the smallest extension of $\sF/\Tors(\sF)$ to a reflexive sheaf. The above mentioned equivalence of reflexive rank $1$ sheaves and Weil divisors is defined verbatim the same 
way as for Cartier divisors. That is, in one direction sections of rank $1$ reflexive sheaves define Weil 
divisors (indeed such sheaves are free at the codimension $1$ points by the classification of finitely generated modules over PID's). In the other direction, $(\sO_X(D))(U)$ is defined with the usual formula, $\{f \in \sK(X)|(f)+D|_U \geq 0\}$. 

Now, we can state the property we would like to use: if $X$ and $Y$ are normal, and $D$ is a Weil divisor on $Y$, then for any integer $e>0$, $\left(F^e\right)^! \sO_X(D) \cong \sO_X(K_X + p^e(D-K_X))$. This is in fact, not hard to show. First, one shows by \autoref{itm:duality} that $f^! \sO_X(D)$ is reflexive. However, then the above isomorphism can be proven on $X_{\reg}$ (by the above unique extension property from $X_{\reg}$). Second, there we obtain the statement by \autoref{itm:f_upper_shriek_twist_line_bundle} together with the fact that $F^* \sL \cong \sL^p$ for any line bundle (, since the gluing functions are raised to $p$-th power).

\end{enumerate}
\end{remark}

\subsection{Applications of duality theory}
\label{sec:F_singularities}
\label{sec:duality_applciation}

In this section we apply the facts recalled in \autoref{sec:duality}, to prove \autoref{thm:F_pure_lc_CY}. In the meanwhile, in the proofs, we also discuss important ideas for the general theory.

In the following lemmas we use the notion of Weil-divisorial sheaves overviewed in point \autoref{itm:f_upper_shriek_twist_divisor} of \autoref{rmk:duality}. Whenever we write $F_* \sO_X(D)$ for some Weil divisor $D$, we mean $F_* (\sO_X(D))$. We omit the parenthesis for simplicity.

\begin{lemma}
\label{lem:splitting_lemma}
Let $X$ be normal variety, and $D$ a Weil divisor on $X$. Then,  there is a one-to-one equivalence between splittings of
\begin{equation}
\label{eq:splitting_lemma:splitting}
\sO_X \hookrightarrow F_* \sO_X(D)
\end{equation}
and between sections $s$ of $F_* \sO_X((1-p)K_X - D)$ satisfying
\begin{equation*}
\xymatrix@C=2pt{
s \ar@{|->}@/^3pc/[rrrrrrrr]^{\ } & \in  F_* \sO_X((1-p)K_X - D) \ar[rrr] &&& F_* \sO_X((1-p)K_X ) \ar[rrr]^-{\Tr_{F, \sO_X}} &&& \sO_X \ni &  1
},
\end{equation*}
given by applying the duality functor $\sHom_X(\_, \sO_X)$. 
\end{lemma}

\begin{proof}
The splitting of \autoref{eq:splitting_lemma:splitting} is given by a diagram as follows.
\begin{equation}
\label{eq:splitting_lemma:splitting_functorial}
\xymatrix{
 \sO_X \ar@/^1.5pc/[rr]^{\Id} \ar[r] & F_* \sO_X(D) \ar[r] & \sO_X 
}
\end{equation}
Note that applying $\sHom_X(\_, \sO_X)$ to $F_* \sO_X(D)$ yields:
\begin{multline}
\label{eq:splitting_lemma:trace_map}
\sHom_X(F_* \sO_X(D),\sO_X) \cong 
\underbrace{F_* \sHom_X(\sO_X(D), F^! \sO_X)}_{\textrm{\eqref{itm:duality} of \autoref{rmk:duality}}} 
\cong 
\underbrace{F_* \sHom_X(\sO_X(D), \sO_X((1-p)K_X))}_{\textrm{\eqref{itm:f_upper_shriek_twist_line_bundle} of \autoref{rmk:duality}}} 
\\ \cong F_* \sO_X((1-p)K_X -D) .
\end{multline}
Hence, by applying the duality functor $\sHom_X(\_, \sO_X)$ to the entire \autoref{eq:splitting_lemma:splitting_functorial}, we obtain that \autoref{eq:splitting_lemma:splitting_functorial} is equivalent to 
\begin{equation}
\label{eq:splitting_lemma:dual}
\xymatrix{
 \sO_X  & \ar[l]  F_* \sO_X((1-p)K_X -D) \ar[l] &  \ar[l] \ar@/_1.5pc/[ll]_{\Id}  \sO_X .
}
\end{equation}
Equation \autoref{eq:splitting_lemma:dual} is then equivalent to the existence of a section of $F_* \sO_X((1-p)K_X -D) $ mapping to $1$ via the map of \autoref{eq:splitting_lemma:trace_map}.
\end{proof}

\begin{remark}
\label{rmk:F_adjunction}
Ideas in the proof of \autoref{lem:splitting_lemma} lead to a more general equivalence between divisors and maps $F_*^e \sL \to \sO_X$, where $\sL$ is a line bundle \cite[Rem 9.5]{Schwede_F_adjunction}. This in turn leads to a more general $F$-adjunction theory \cite{Schwede_F_adjunction}, which  turns out to be using the same different as the usual adjunction theory of the Minimal Model Program \cite[Thm 5.3]{Das_On_strongly_F_regular_inversion_of_adjunction}.
\end{remark}

\begin{lemma}
\label{prop:smaller_than_p_to_the_e}
Let $X:=\Spec R$ for a DVR $R$, and let $D$ be the divisor defined by the local parameter. If $\sO_X \hookrightarrow F_* \sO_X(rD)$ splits for some integer $r>0$, then $r <p$. 

\end{lemma}

\begin{proof}
Let $t \in R$ be the local parameter. For any injection $\iota: \sO_X \hookrightarrow \sE$ into a free coherent sheaf (which is equivalent to torsion-free here), if $\iota$ splits, then there cannot be $s \in \sE$ such that $\iota(1)= t s$. Indeed, if there was such an $s$, then for the splitting $\psi: \sE \to \sO_X$ we would have $t\psi(s)=\psi(\iota(1))=1$. This is impossible. 

On the other hand, if $r \geq p$, then $\sO_X \hookrightarrow F_* \sO_X(rD)$ factors as 
\begin{equation*}
\xymatrix{
\sO_X \ar@{^(->}[r]_{\alpha} \ar@/^1.5pc/[rr]^-{\iota} & \sO_X(D) \ar@{^(->}[r]_-{\beta} & \sO_X(D) \otimes F_* \sO_X((r-p)D) \cong F_* \sO_X(rD). 
}
\end{equation*}
Therefore, $\alpha(1)=t s'$ for some $s ' \in \sO_X(D)$, and then for $s:= \beta (s')$ we have $\iota(1)= t s$.
\end{proof}

\begin{proposition}
\label{prop:F_pure_boundary}
Let $X$ be either affine and $F$-pure or projective and globally $F$-pure, and normal in either case.  
Then, there is an effective $\bQ$-divisor $\Delta$ on $X$ such that $(p-1)(K_X + \Delta) \sim 0$ (where $\sim$ means actual linear equivalence, not only $\bQ$-linear equivalence), and furthermore, the natural map 
\begin{equation}
\label{eq:F_pure_boundary:splitting}
\sO_X \hookrightarrow F_* \sO_X((p-1)\Delta)
\end{equation}
splits. 
\end{proposition}

\begin{proof}
Let $\psi : F_* \sO_X \to  \sO_X$ be the splitting given by the $F$-pure assumption. 
%
Let $s$ be the section given by \autoref{lem:splitting_lemma}, with $D=0$.
Setting then $D:=V(s)$ (where $s$ regarded as an element of $H^0(X, \sO_X((1-p)K_X))$ and $\Delta:=\frac{D}{p-1}$, we have
\begin{equation*}
(p-1)(K_X + \Delta) \sim (p-1)K_X + D \sim (p-1)(K_X -K_X)=0.
\end{equation*}
Furthermore, since $s$ maps to $1$ via $\Tr_{F,\sO_X}$, so does $1$ via the following composition. 
\begin{equation*}
F_* \sO_X((1-p)(K_X + \Delta))  \to F_* \sO_X((1-p)K_X) \to \sO_X
\end{equation*}
Then, \autoref{lem:splitting_lemma} applied with $D=(p-1)\Delta$ yields the splitting of \autoref{eq:F_pure_boundary:splitting}.
\end{proof}

\begin{proof}[Proof of \autoref{thm:F_pure_lc_CY}]
 Let $f : Y \to X$ be a  birational, proper map from a normal variety. Let $\Delta$ be the $\bQ$-divisor guaranteed by \autoref{prop:F_pure_boundary}, and let $s \in H^0(X, F_* \sO_X((1-p)(K_X + \Delta))) \cong H^0(X,  \sO_X((1-p)(K_X + \Delta)))$ be the section guaranteed by \autoref{lem:splitting_lemma} (by setting $D:= (p-1) \Delta$). That is, $s$ surjects onto $\sO_X$ via $F_* \sO_X((1-p)(K_X + \Delta)) \to \sO_X$. Define the $\bQ$-divisor $\Gamma$ via the following equality of Cartier divisors, assuming that $f_* K_Y = K_X$.
\begin{equation*}
(1-p)f^* (K_X + \Delta) = (1-p)(K_Y + \Gamma)
\end{equation*}
 Hence $s$ induces a section $t$ of $\sO_X((1-p)(K_Y + \Gamma))$. Consider now the following diagram.
\begin{equation*}
\xymatrix@C=0pt{
tu \in  \ar@/^3pc/[rrrrrr]^{\ } &  F_* f_* \sO_Y ((1-p)(K_Y + \Gamma)) \cong f_* F_* \sO_Y ((1-p)(K_Y + \Gamma)) \ar[rrrr]  \ar@{=}[d] &&&& f_* \sO_Y \ar@{=}[d] & \ni 1  \ar@{<->}[d] \\
s \in  \ar@/_3pc/[rrrrrr]_{\ } & F_* \sO_X ((1-p)(K_X + \Delta)) \ar[rrrr] &&&& \sO_X  & \ni 1
}
\end{equation*}
The diagram commutes with the advanced setup of dualizing theory, since it commutes over the open set $U \subseteq X$ over which $f$ is an isomorphism, and furthermore $\codim_X X \setminus U \geq 2$ (see \autoref{rmk:duality}.\autoref{itm:f_upper_shriek} for a discussion on this). In our setup, based on only facts from \cite{Hartshorne_Algebraic_geometry}, trace maps are only defined up to pre-multiplication by a unit in the source. In our case this means a unit $u \in H^0(Y, F_* \sO_Y)$. However, then $tu$ will be taken to $1$ over $U$, and then using the above codimension condition, also globally. Using now \autoref{lem:splitting_lemma} we obtain that the map 
\begin{equation*}
\sO_Y \hookrightarrow F_* \sO_Y ((p-1) \Gamma)
\end{equation*}
admits a splitting. However, then \autoref{prop:smaller_than_p_to_the_e} implies that $\coeff_E ((p-1) \Gamma) \leq p-1$ at every prime divisor $E$ of $Y$, which is equivalent to saying that $\coeff_E\Gamma \leq 1$. This is exactly the condition of log canonicity. Hence we have finished our proof. 
\end{proof}

\section{Newer methods - finding sections}
\label{sec:newer_methods}

Many results in higher dimensional algebraic geometry can be shows by exhibiting sections of line bundles with certain positivity. In characteristic 0,  this is usually done by considering exact sequences of the type 
\begin{equation*}
H^0(X, \sL) \to H^0(Z, \sL) \to H^1(X, \sI_Z \otimes \sL),
\end{equation*}
(where $Z$ is a closed subvariety of $X$, and $\sL$ is a line bundle), and using that $H^1(X, \sI_Z \otimes \sL)$ is zero by some vanishing theorem. Unfortunately, these vanishing theorems fail collectively in positive characteristic. 
Having introduced the basic notions of the Frobenius method in \autoref{sec:basic_notions}, as well as having reviewed their relation to the notions of birational geometry, we proceed here with presenting the recent methods to finding sections, and hence circumventing the above mentioned failure of vanishing theorems. 

We state each statement in a special case, and we refer to the original sources for the full generality. On the other hand, we also present proofs, which hopefully in these specialized setups are easier to follow than their original appearances. We also present some  sample applications in \autoref{cor:pushforward}, \autoref{cor:pushforward_geometric}, \autoref{cor:sample} and \autoref{cor:A_1}. For a list of further, mostly more involved, applications see the articles surveyed in \autoref{sec:application}. 

\subsection{Cartier modules}
\label{sec:Cartier_modules}


In characteristic $0$, the theory of the classification of algebraic varieties largely depends on Hodge theory, so on objects that are either $\sD$-modules or are closely related to $\sD$-modules. For example, one can think about Kodaira vanishing and its souped up versions, or about the theory of variations of Hodge structures, etc. This $\sD$-module theoretic point of view was unified by Saito's celebrated theory of Hodge modules (e.g., \cite{Saito_A_young_person_s_guide_to_mixed_Hodge_modules}). In this point of view, roughly all the above mentioned Hodge theoretic foundations of higher dimensional algebraic geometry are of $\sD$-module theoretic origin. 

The main feature of $\sD$-modules on varieties over $\bC$ is the Riemann-Hilbert correspondence, that is, they are equivalent to constructible sheaves in the adequate sense \cite{Deligne_Equations_differentielles_a_points_singuliers_reguliers,Kashiwara_The_Riemann-Hilbert_problem_for_holonomic_systems}. Surprisingly, there have been quite a few versions of this correspondence shown in positive characteristic for \'etale, constructible $\bF_p$ sheaves in the last one and a half decades \cite{Emerton_Kisin_The_Riemann-Hilbert_correspondence_for_unit_F-crystals,Bockle_Pink_Cohomological_theory_of_crystals_over_function_fields,Blickle_Bockle_Cartier_modules_finiteness_results}. The latter of these result uses the category of \emph{Cartier modules}, or rather a localization of it along a Serre subcategory. Then, taking into account the  characteristic zero phenomena mentioned in the previous paragraph, it is perhaps not an enormous surprise that Cartier modules became very important recently for the algebro 
geometric implications of \autoref{qtn:module}. Before proceeding to the actual definition, we also note that if one wants to avoid constructible sheaves, there is also a direct relation, although not one-to-one in any sense, between Cartier modules and $\sD$-modules \cite[Chapter 1]{Blickle_The_D_module_structure_of_R_F_modules}.

\begin{definition}
A \emph{Cartier module} on $X$ is a triple $(\sM, e, \phi)$, where $\sM$ is a coherent sheaf on $X$, $e>0$ is an integer, and $\phi : F^e_* \sM \to \sM$ is homomorphism of coherent sheaves. 

 In this setting $\phi^s : F^{s\cdot e}_* \sM \to \sM$ is defined as the composition of the following homomorphisms.
\begin{equation*}
\xymatrix@C=50pt{
F^{s \cdot e}_* \sM \ar[r]^{F^{(s-1) \cdot e}_* \phi } & F^{(s-1) \cdot e}_* \sM \ar[r]^{F^{(s-2) \cdot e}_* \phi} &  \dots \ar[r]^{F^{e}_* \phi} & F^e_* \sM \ar[r]^\phi & \sM 
}
\end{equation*}

\end{definition}

\begin{example}
\label{ex:Cartier_module_P_1}
It is an easy exercise to  show that $F_* \sO_{\bP^1} \cong \sO_{\bP^1} \oplus \left(  \sO_{\bP^1}(-1)^{\oplus (p-1)} \right)$. In particular, by projecting to the first factor, we obtain a Cartier module structure on $\sO_{\bP^1}$.
\end{example}

\begin{example}
\label{ex:Cartier_module_dualizing}
The most typical example of Cartier modules are coming from dualizing sheaves. Indeed, using the notations of \autoref{rmk:duality}, $\Tr_F : F_* \omega_X \to \omega_X$ is a Cartier module. Furthermore, if $f : X \to Y$ is a proper morphism, one can endow also $R^i f_* \omega_X$  with a Cartier module structure. Indeed, by pushing forward $\Tr_F$ we obtain a homomorphism $\psi : R^i f_* F_* \omega_X \to R^i f_* \omega_X$.  Then using that $f \circ F= F \circ f$ and that $F$ is an affine morphism, we obtain that $F_* R^i f_* \omega_X \cong R^i f_* F_* \omega_X$. Composing this isomorphism with $\psi$ yields the desired Cartier module structure on $R^i f_* \omega_X$. 
\end{example}

\begin{proposition} \cite[Lemma 13.1]{Gabber_Notes_on_some_t_structures}  \cite[Prop 1.11]{Hartshorne_Speiser_Local_cohomological_dimension_in_characteristic_p} \cite{Lyubeznik_F_modules_applications_to_local_cohomology_and_D_modules_in_characteristic_p_greater_than_0} \cite[Proposition 8.1.4]{Blickle_Schwede_p_to_the_minus_one_linear_maps_in_algebra_and_geometry}
\label{prop:Cartier_module_stabilizes}
If  $(\sM,\phi, e)$ is a Cartier module, then the descending chain $\sM \supseteq \im \phi \supseteq \im \phi^2 \supseteq \dots $ stabilizes.
\end{proposition}

\begin{definition} 
The stable image of \autoref{prop:Cartier_module_stabilizes} is denoted by $\sigma (\sM)$. 
\end{definition}

%

The above definition allows us to define an ideal which measures how much a singularity is not $F$-pure.

\begin{definition}
Assume that $\omega_X$ is a line bundle. Then $\sigma(X) \subseteq \sO_X$ is the ideal, such that $\sigma(X) \otimes \omega_X$ is the stable image of the Cartier module $(\omega_X, \Tr_F)$. 
\end{definition}

\begin{proposition}
\label{prop:sigma}
$\sigma(X) = \sO_X$ if and only if $X$ is $F$-pure.
\end{proposition}

\begin{proof}
We may assume that $X$ is affine. The equality $\sigma(X)=\sO_X$ is equivalent to $F_*^e \omega_X \to \omega_X$ being surjective. By duality this is equivalent (as in \autoref{lem:splitting_lemma}) to the splitting of $\sO_X \to F^e_* \sO_X$ for all integers $e >0$. This can be seen to be equivalent to the splitting of only $\sO_X \to F_* \sO_X$ \cite[Exc 2.8]{Schwede_Tucker_A_survey_of_test_ideals}, which is exactly the definition of $F$-purity. 
\end{proof}

\subsection{Positivity results on Cartier modules}
\label{sec:positivity_Cartier_modules}

The main tool for finding sections using Cartier modules is the following theorem. 

\begin{theorem} (c.f., \cite{Keeler_Fujita_s_conjecture_and_Frobenius_amplitude,Schwede_A_canonical_linear_system,Patakfalvi_On_subadditivity_of_Kodaira_dimension_in_positive_characteristic_over_a_general_type_base})
\label{thm:Cartier_module_globally_generated}
If $(\sM, \phi , e)$ is a coherent Cartier module, $\sA$ is an ample globally generated line bundle and $\sH$ is an arbitrary ample line bundle on $X$, then $\sigma(\sM) \otimes \sA^{\dim X} \otimes \sH$ is globally generated. 
\end{theorem}

\begin{proof}
Set $n:= \dim X$. By \autoref{prop:Cartier_module_stabilizes},  $\phi^s : F^{s \cdot e}_*  \sM \to \sigma(\sM)$ is surjective  for every integer $s  \gg 0$. Therefore it is enough to prove that $F^{s \cdot e}_* \sM \otimes \sA^n \otimes \sH$ is globally generated for every $e \gg 0$. Hence, by \cite[Theorem 1.8.5]{Lazarsfeld_Positivity_in_algebraic_geometry_I} it is enough to prove that for every $e \gg 0$ and $i>0$,
\begin{equation*}
H^i(X,F^{s \cdot e}_* \sM  \otimes \sA^{n-i} \otimes \sH) = 0 .
\end{equation*}
However,
\begin{multline*}
H^i\left(X,F^{s \cdot e}_* \sM \otimes \sA^{n-i} \otimes \sH\right) 
\cong \underbrace{H^i\left(X,  F^{s \cdot e}_*  \left( \sM  \otimes F^{s \cdot e,*} \sA^{n-i} \otimes  F^{s \cdot e,*} \sH \right) \right)}_{\textrm{projection formula}} 
\\ \cong \underbrace{H^i\left(X, \sM  \otimes F^{s \cdot e,*} \sA^{n-i} \otimes  F^{s \cdot e,*} \sH\right) }_{\textrm{$F$ is affine}}  
\cong H^i \left(X, \sM  \otimes  (\sA^{n-i} \otimes  \sH)^{p^{s \cdot e}} \right).
\end{multline*}
Since $n -i \geq 0$ and both $A$ and $H$ are ample, Serre-vanishing concludes our proof.
\end{proof}

\autoref{thm:Cartier_module_globally_generated} says that (non-nilpotent) Cartier modules cannot be arbitrarily negative:

\begin{corollary}
Every Cartier module structure on $\sO_{\bP^n}(d)$ is nilpotent for $d< -n -1$ (nilpotent means that the stable submodule is zero).
\end{corollary}

\begin{proof}
Assume there is a non-nilpotent Cartier module structure  on $\sO_{\bP^n}(d)$. Then $\sigma(\sO_{\bP^n}(d))$ is a non-zero coherent submodule of $\sO_{\bP^n}(d)$. In particular, it is a line bundle $\sO_{\bP^n}(d')$ for some $d' \leq d$. Apply now \autoref{thm:Cartier_module_globally_generated} with $\sA= \sH := \sO_{\bP^n}(1)$. Then we obtain that $\sO_{\bP^n}(d' +n +1)$ is nef. Hence $d' + n+1 \geq 0$, and then $d \geq d' \geq -n -1$.
\end{proof}

Here is our baby application of \autoref{thm:Cartier_module_globally_generated}. It is a positive characteristic analogue of the famous characteristic zero statement of Fujita that $f_* \omega_{X/Y}$ is a nef vector bundle for a fibration over a curve \cite{Fujita_On_Kahler_fiber_spaces}.  This particular statement applies for example to any Gorenstein degeneration of a smooth, projective variety with full rank Hasse-Witt matrix (see the definition in the statement). Examples of such varieties are ordinary abelian varieties, or ordinary K3 surfaces. For other results in this direction, using related methods, see \cite{Patakfalvi_Semi_positivity_in_positive_characteristics,Patakfalvi_On_subadditivity_of_Kodaira_dimension_in_positive_characteristic_over_a_general_type_base,Ejiri_Weak_positivity_theorem_and_Frobenius_stable_canonical_rings_of_geometric_generic_fibers,Ejiri_Iitaka_s_C_n_m_conjecture_for_3-folds_in_positive_characteristic}. After \autoref{cor:pushforward} we give a sample geometric application in 
\autoref{cor:pushforward_geometric}.

\begin{corollary}
\label{cor:pushforward}
Assume the base field is algebraically closed. Let $f : X \to T$ be a Gorenstein, flat,  projective morphism of pure dimension $n$ to a smooth, projective curve, such that general fibers $X_t$ are smooth varieties with full rank (non-zero) Hasse-Witt matrix, where the latter is the matrix (in any basis) of the action of Frobenius on $H^n\left(X_t, \sO_{X_t}\right)$. Then $f_* \omega_{X/T}$ is nef.
\end{corollary}

\begin{proof}

First, note that $f_* \omega_{X/T}$ is torsion-free, since it is the push-forward of a torsion-free sheaf. Since $T$ is a smooth curve, this means that $f_* \omega_{X/T}$ is in fact locally free. Then,  by duality \cite{Kleiman_Relative_duality_for_quasicoherent_sheaves} one obtain that $R^n f_* \sO_X \cong \sHom_T( f_* \omega_{X/T}, \sO_T)$ is also locally free. 

Consider now the commutative diagram on the left for $t \in T$ a general closed point, which induces then the other diagram on the right.
\begin{equation*}
\xymatrix{
X \ar[r]^F & X  \\
X_t \ar[r]^F \ar@{^(->}[u] & X_t \ar@{^(->}[u]
} \qquad 
\raisebox{-20pt}{$\Rightarrow$}
\qquad \xymatrix{
F_* \sO_X \ar@{->>}[d]  & \ar[l] \sO_X \ar@{->>}[d] \\
F_* \sO_{X_t} & \ar[l] \sO_{X_t}
}
\end{equation*}
Let $n$ be the relative dimension of $f$, and apply $R^n f_*( \_ )$ to the latter diagram. This yields
\begin{equation}
\label{eq:pushforward:derived_direct}
\xymatrix{
F_*  R^n f_* (\sO_X) \cong R^n f_* (F_* \sO_X) \ar[d]  & \ar[l] R^n f_* (\sO_X) \ar[d] \\
H^n\left(X_t,F_* \sO_{X_t}\right) & \ar[l] H^n \left(X_t,\sO_{X_t}\right)
}.
\end{equation}
By cohomology and base-change \cite[Thm III.12.11]{Hartshorne_Algebraic_geometry}, $(R^n f_* \sO_X) \otimes k(t) \cong H^n (X_t, \sO_{X_t})$ via right vertical map of \autoref{eq:pushforward:derived_direct}. Furthermore, the bottom horizontal arrow is exactly the Frobenius action on $H^n (X_t, \sO_{X_t})$, which is then injective. It follows then that $(R^n f_* (\sO_X)) \otimes k(t) \to (F_*  R^n f_* (\sO_X)) \otimes k(t)$ is injective. Since $R^n f_* \sO_X$ is locally free, it follows that $R^n f_* (\sO_X) \to F_*  R^n f_* (\sO_X)$ is already injective. In particular, by applying the  duality functor $\sHom_T(\_, \omega_T)$ we obtain that $F_* f_* \omega_X \to f_* \omega_X$ is generically surjective. This then implies, using \autoref{prop:Cartier_module_stabilizes}, that {\it $\sigma(f_* \omega_X) \subseteq f_* \omega_X$ has full rank.}

Now, we have to prove that for every quotient bundle $f_* \omega_{X/T} \twoheadrightarrow \sQ$, $\deg \sQ \geq 0$. In fact, \emph{it is enough to prove that there is an integer $d$ (possibly negative) such that for every integer $e \geq 0$, $\deg \left( F^e \right)^* \sQ \geq d$}. Indeed, if we prove this then $p^e \det \sQ = \det \left( F^e \right)^* \sQ \geq d$ for every integer $e \geq 0$, and letting $e \to \infty$ proves our claim. 

Denote by $T^e$ the source of $F^e$ to distinguish between the source and the target. Then by flat base-change ($F$ is flat, since $T$ is smooth), $\left( F^e \right)^* f_* \omega_{X/T} = \left( f_{T^e} \right)_* \omega_{X_{T^e}/T^e}$. Hence, $\left( F^e \right)^* \sQ$ is a quotient of $\left( f_{T^e} \right)_* \omega_{X_{T^e}/T^e}$. Hence, since $f_{T^e}$ satisfies all the assumptions that we have on $f$, it is enough to show the above statement only for $f$. That is, \emph{it is enough to show that there is an integer, $d$ depending only on $T$ and $\rk(\sQ)$ such that $\deg \sQ \geq d$}.

Choose now any line bundle $\sA$ of degree $2g(T)$ and any degree $1$ line bundle $\sH$. Then $\sA$ is ample and globally generated by \cite[Cor IV.3.2]{Hartshorne_Algebraic_geometry}. Hence, by \autoref{thm:Cartier_module_globally_generated}, $\sigma(f_* \omega_X) \otimes \sL \otimes \sH$ is globally generated, which in turn implies that $(f_* \omega_{X/Y}) \otimes \omega_Y \otimes \sA \otimes \sH$ is generically globally generated, hence so is $\sQ \otimes \omega_Y \otimes \sA \otimes \sH$. In particular, the latter sheaf has degree at least zero, and hence 
\begin{equation*}
\deg \sQ \geq - \rk(\sQ) \deg (\omega_Y \otimes \sA \otimes \sH) = - \rk (\sQ)(4 g(T)-1).
\end{equation*}
This concludes our proof. 
\end{proof}

There are many more geometric statements that one can deduce from \autoref{cor:pushforward} and the above mentioned related results. We give only one example, which was chosen admittedly arbitrarily.

\begin{corollary}
\label{cor:pushforward_geometric}
Let $f : X \to T$ be a smooth Calabi-Yau $3$ fold (where we only use that $\omega_X \cong \sO_X$) fibered over a curve with generic fibers being smooth ordinary K3 surfaces or ordinary abelian surfaces (where we only use that $\omega_{X_t} \cong \sO_{X_t}$ and that the Hasse-Witt matrix has full rank).  Then $g(T) \leq 1$.
\end{corollary}

\begin{proof}
By \autoref{cor:pushforward}, $f_* \omega_{X/T}$ is nef and non-zero. Furthermore, since the general fibers are Calabi-Yau varieties for which $\omega_{X_t} \cong \sO_{X_t}$, $f_* \omega_{X/T}$ is in fact of rank $1$ and hence a line bundle. Additionally, then by the same reason $f_* f^* \omega_{X/T} \to \omega_{X/T}$ is an isomorphism at the generic fiber of $f$. So, if $D$ is a divisor of the line bundle $f_* \omega_{X/T}$, then for some effective divisor $\Gamma$ on $X$, $K_{X/T} \sim f^* D + \Gamma$, where $\deg D \geq 0$. Hence,
\begin{equation*}
\underbrace{0 \sim K_X}_{\textrm{Calabi-Yau assumption on $X$}} = K_{X/T} + f^* K_T \sim f^* (D  + K_T ) + \Gamma
\end{equation*}
Now, if $g(T) >1$, then $\deg (K_T +D) >0$, and so  for the Kodaira dimension, $\kappa(f^* (D+ K_T) + \Gamma) \geq 1$ holds. This  contradicts the  $K_X \sim 0 $ in the above equation.
\end{proof}

\subsection{Frobenius stable sections and the lifting lemma}
\label{sec:Frobenius_stable}

A twist of the Cartier module $F_*  \omega_X \to \omega_X$, explained in \autoref{ex:Cartier_module_dualizing}, shows up in adjunction situations \cite{Schwede_F_adjunction}. This led eventually to another method of finding sections, which in turn led to the advances in 3-fold Minimal Model Program \cite{Hacon_Xu_On_the_three_dimensional_minimal_model_program_in_positive_characteristic}. We present here the simplified version of the theory.

\begin{definition}
\label{def:S_0}
Let $X$ be a smooth projective variety, $D$ an effective divisor  and $\sL$ a line bundle on it. Then 
\begin{equation*}
S^0(X, D;\sL) : = \im \left( H^0\left(X, \sL \otimes  F^e_* \sO_X((1-p^e)(K_X +D)) \right) \to H^0(X, \sL)  \right)
\end{equation*}
for any integer $e \gg 0$. Here, the map is induced from $\Tr_{F^e, \sO_X}$ by precomposing with ``multiplication'' by $(p^e-1)D$, then tensoring with $\sL$, and finally applying $H^0(\_)$ to it. 

The above maps factor through each other, and hence the images form a descending sequence of subspaces in a finite dimensional vector space. Hence, for $e \gg 0$ the images stabilize,  and hence the definition makes sense.
\end{definition}

\begin{notation}
In the above situation, if $D=0$, then instead of $S^0(X, 0; \sL)$ we sometimes write $S^0(X; \sL)$.
\end{notation}

\begin{theorem} \cite[Prop 5.3]{Schwede_A_canonical_linear_system}
\label{thm:lifting}
Let $X$ be a smooth projective variety, $D$ a smooth effective divisor and $\sL$ a line bundle, with $L$ a divisor corresponding to $\sL$. If $L-K_X -D$ is ample, then there is a natural surjection
\begin{equation*}
S^0(X, D; \sL ) \twoheadrightarrow S^0(D; \sL|_D).
\end{equation*}
\end{theorem}

\begin{proof}
Consider the following commutative diagram.
\begin{equation*}
\xymatrix{
0 \ar[r] & \sO_X(-D) \ar[d] \ar[r] & \sO_X \ar[d] \ar[r] & \sO_D \ar[d] \ar[r] & 0 \\
0 \ar[r] & F_*^e \sO_X(-D) \ar[r] & F_*^e \sO_X \ar[r] & F_*^e \sO_D \ar[r] & 0 
}
\end{equation*}
Applying $\sHom(\_, \sO_X) \otimes \sO_X(-D)$ to the above diagram (using duality and the projection formula a few times) we obtain
\begin{equation*}
\xymatrix{
\sO_X & \ar[l] \sO_X(-D) & \ar[l] 0  \\
F_*^e \sO_X((1-p^e)(K_X +D)) \ar[u] & \ar[l] \ar[u] F_*^e \sO_X((1-p^e)K_X - p^eD) & \ar[l] 0
}.
\end{equation*}
Completing with the cokernels we obtain the following diagram:
\begin{equation*}
\xymatrix{
0 & \sO_D \ar[l] & \ar[l] \sO_X & \ar[l] \sO_X(-D) & \ar[l] 0  \\
0 & F^e_* \sO_D((1-p^e)K_D) \ar[l] \ar[u] & \ar[u] \ar[l] F_*^e \sO_X((1-p^e)(K_X +D)) & \ar[u] \ar[l] F_*^e \sO_X((1-p^e)K_X - p^eD) & \ar[l] 0
}
\end{equation*}
One can show that the left-most vertical arrow is the trace map (this follows from the setup of the duality theory in \cite{Hartshorne_Residues_and_duality}, or from the fact that the trace map at the Gorenstein points is the local generator of $\sHom_X(F_*^e \sO_X((1-p^e)K_X), \sO_X)$, which can be proven by an easy duality argument). Hence, we see that if we tensor this diagram with $\sL$ and then take cohomology long exact sequence of both rows we obtain for each integer $e>0$ the following diagram where the left column is exact:
\begin{equation*}
\xymatrix{
H^0(X, \sL \otimes F_*^e \sO_X((1-p^e)(K_X +D)) ) \ar[r] \ar[d] & H^0(X, \sL) \ar[d] \\
H^0(D, \sL|_D \otimes F^e_* \sO_D((1-p^e)K_D) ) \ar[r] \ar[d] & H^0(D, \sL|_D)   \\
H^1(X ,  \sL \otimes F_*^e \sO_X((1-p^e)K_X - p^eD) ) & 
} 
\end{equation*}
In particular, we see that to prove the surjectivity claimed in the statement, it is enough to show that for all $e \gg 0$, 
\begin{multline*}
0 = H^1(X , \sL \otimes  F_*^e \sO_X((1-p^e)K_X - p^eD) ) = H^1(X ,  \sO_X(p^e L - (1-p^e)K_X - p^eD) )
\\ = H^1(X ,  \sO_X(p^e (L - K_X -D) +K_X) ).
\end{multline*}
This vanishing is satisfied by Serre vanishing, using the $L-K_X -D$ ample assumption. 
\end{proof}

\begin{remark}
\autoref{thm:lifting} is the simplified version of \cite[Thm 5.3]{Schwede_A_canonical_linear_system}. The latter is presented in the more general framework of pairs, using the adjunction theory mentioned in \autoref{rmk:F_adjunction}.
\end{remark}

\subsection{Bounding $F$-stable sections}
\label{sec:bounding_Frobenius_stable}

To use \autoref{thm:lifting} for finding sections we need another ingredient guaranteeing that $S^0(D; \sL|_D)$ is big enough. 
In dimension $1$, i.e., if $X$ is a smooth, projective curve, Tango showed that $H^0(X,\sL)= S^0(X, \sL)$ for $\deg \sL > \left\lfloor (2g(X)-2) \frac{p+1}{p} \right\rfloor$ \cite[Lemma 10- Lemma 12]{Tango_On_the_behavior_of_extensions_of_vector_bundles_under_the_Frobenius_map}. The following is a sample baby application of this bound. For stronger statements in this direction, we refer to \cite{Ekedahl_Canonical_models_of_surfaces_of_general_type_in_positive_characteristic,Shepherd_Barron_Unstable_vector_bundles_and_linear_systems_on_surfaces_in_characteristic_p}, although these use different techniques.

\begin{corollary}
\label{cor:sample}
Let $X$ be a smooth, projective surface with $K_X$ ample. If $mK_X \sim C$ for some smooth curve $C$ and integer $m \geq 2$, then $|m(m+1)K_X|$ is a free linear system.
\end{corollary}

\begin{proof}
Let $L:=m(m+1)K_X \sim m (K_X +C)$ and $\sL:= \sO_X(L)$. Then $L|_C \sim m K_C$. In particular, $S^0(C, \sL|_C)= H^0(C, \sL|_C)$ by Tango's bound. Hence, by \autoref{thm:lifting}, $S^0(X, C;\sL) \subseteq H^0(X, \sL)$ surjects onto $H^0(C, \sL|_C)$. In particular, using that $mK_C$ is free for $m \geq 2$, we obtain that $L$ is free at every point of $C$. On the other hand, it is also free at the points of $X \setminus C$, since $(m+1)C \sim L$. 
\end{proof}

Unfortunately, in higher dimensions one cannot hope for a result similar to Tango's, that is, a statement saying that for ample line bundles $\sL$ of big enough volume $S^0(X, \sL) = H^0(X, \sL)$ holds, where the bounds depend only $\vol(X)$. Indeed, one can easily find a counterexample amongst surfaces of the form $C \times D$,  for some curves $C$ and $D$, and line bundles $\sL$ of the form $\sN \boxtimes \sM$. Then, $S^0(X,\sL) = S^0(C, \sM) \otimes S^0(D, \sN)$ (c.f., \cite[Lemma 2.3.1]{Patakfalvi_Semi_positivity_in_positive_characteristics}), so if $S^0(C, \sM) \neq H^0(C, \sM)$ (which is not hard to arrange, e.g., \cite[Example 2-3]{Tango_On_the_behavior_of_extensions_of_vector_bundles_under_the_Frobenius_map}), then $S^0(X,\sL) \neq H^0(X,\sL)$. On the other hand, by increasing the degree of $\sN$, $\vol(\sL)$ goes to infinity. Hence, so far the best known statement is \cite[Cor 2.23]{Patakfalvi_Semi_positivity_in_positive_characteristics}, the simplified version of which is as follows.

\begin{theorem} 
\label{thm:S_0_equals_H_0} \cite[Cor 2.23]{Patakfalvi_Semi_positivity_in_positive_characteristics}
Let $X$ be a smooth, projective variety, $D$ a smooth effective divisor, and $\sL$ an ample line bundle on $X$. Then, for every $n \gg 0$, 
\begin{equation*}
H^0(X, \sL^n) = S^0(X,D;\sL^n).
\end{equation*}

\end{theorem}

\begin{proof}
By the functoriality of trace maps, we have an  infinite diagram, where the compositions of any successive arrows is $F^{i}_* (\Tr_{F^j})$ for some $i, j$. 
\begin{equation*}
\xymatrix{
  \dots \ar[r] &  F^e_* \omega_X \ar[r] & F^{e-1}_* \omega_X \ar[r] & \dots \ar[r] & F_* \omega_X \ar[r] & \omega_X
}
\end{equation*}
Twisting this diagram with $\omega_X^{-1}$ and precomposing each map with ``multiplication'' by $(p^e-1)D$, we obtain the following other diagram, where all maps  are induced from $\Tr_{F, \sO_X}$ by  precomposing with ``multiplication'' by $(p-1)D$, then tensoring with $\sO_X((1-p^{e-1})(K_X +D))$, and applying $F^{e-1}_*(\_)$.
\begin{multline*}
  \dots \to  F^e_* \sO_X \left(\left(1-p^e\right)(K_X +D)\right) \to F^{e-1}_* \sO_X\left(\left(1-p^{e-1}\right)(K_X +D)\right) \to \dots  \\
\to  F_* \sO_X((1-p)(K_X +D)) \to \sO_X
\end{multline*}
By tensoring the latter diagram with $\sL^n$ and applying $H^0(X, \_)$, we see that it is enough to find an integer $n_0$ such that the maps obtained this way are all surjective for $n \geq n_0$. That is we want to show that for $n \geq n_0$ and $e >0$ the map 
\begin{equation}
\label{eq:want_surjection}
  H^0\left(X, \sL^n \otimes F^e_* \sO_X((1-p^e)(K_X +D)) \right) \to  H^0\left(X, \sL^n \otimes F^{e-1}_* \sO_X((1-p^{e-1})(K_X +D)) \right) 
\end{equation}
is surjective. However, by the trace map remark above,  this map comes from the exact sequence
\begin{equation*}
\xymatrix{
0 \ar[r] & \sB \ar[r] & F_* \sO_X((1-p)(K_X +D)) \ar[r]^-{\Tr_{F,\sO_X}} & \sO_X \ar[r] & 0.
}
\end{equation*}
In particular, for surjectivity of \autoref{eq:want_surjection}, it is enough to show the following (here $L$ is a divisor of $\sL$):
\begin{multline*}
0 = H^1\left(X, \sL^n \otimes F^{e-1}_* \left( \sO_X((1-p^{e-1})(K_X +D)) \otimes \sB \right) \right)  
\\ \cong
\underbrace{H^1\left(X,  F^{e-1}_* \left( \sO_X((1-p^{e-1})(K_X +D) + p^{e-1} n L ) \otimes \sB_X \right) \right) }_{\textrm{projection formula}}
\\ \cong
\underbrace{H^1\left(X,  \sO_X((1-p^{e-1})(K_X +D) + p^{e-1}n L )  \otimes \sB_X  \right) }_{\textrm{$F^{e-1}$ is affine}}
\\ \cong
H^1\left(X,  \sO_X((p^{e-1}-1)(nL - K_X +D) +  nL )  \otimes \sB_X  \right) 
\end{multline*}
By Fujita vanishing \cite{Fujita_Vanishing_theorems_for_semipositive_line_bundles}, there is an $n_0$, such that the above holds for each integer $e>0$ and $n \geq n_0$. 
\end{proof}

The following is a sample application of \autoref{thm:S_0_equals_H_0}. We note that the surface case of the statement can also be shown using vanishing theorems available for surfaces (e.g. \cite{Hara_Classification_of_two_dimensional_F_regular_and_F_pure_singularities} \cite{Patakfalvi_Schwede_Tucker_Notes_for_the_workshop_on_positive_characteristic_algebraic_geometry}). However, there are no such vanishing theorems in higher dimensions, and hence we are not aware of a significantly different method in the general case. We also note that \autoref{cor:A_1} is an instance of the general phenomenon that some log canonical singularities are $F$-pure. This works so much in dimension $2$ that in fact all surface Kawamata log terminal singularities are $F$-pure if $p>5$ \cite{Hara_Classification_of_two_dimensional_F_regular_and_F_pure_singularities}.

\begin{corollary}
\label{cor:A_1}
An $A_1$ rational double point surface singularity is $F$-pure. More generally, if $x \in X$ is s singularity, such that
\begin{enumerate}
\item  $X$ is Gorenstein, canonical
\item  $X$ admits a resolution of singularities  $f : Y \to X$, such that $E = \Exc(f)$, where
\begin{enumerate}
\item \label{itm:A_1:F_split} $E$ is globally $F$-split and (smooth) Fano (say $E \cong \bP^n$), and
\item \label{itm:A_1:max_ideal} $\sO_Y \cdot f^{-1}m_{x,X} = \sI_{E,Y}$ (the pullback of the ideal of $x$ is the ideal of $E$),
\end{enumerate}
\end{enumerate}
then $X$ is $F$-pure at $x$. 
\end{corollary}

\begin{proof}
We may assume that $X$ is affine and local. By the proof of \autoref{prop:sigma}, we see that $X$ is $F$-pure at $x$  if and only if the trace map $F^e_* \omega_X \to \omega_X$ is surjective for every integer $e \gg 0$. Then, by twisting with $\omega_X^{-1}$, this is equivalent to asking that $F^e_* \sO_X((1-p^e)K_X) \to \sO_X$ is surjective.  

Consider now the following commutative diagram (c.f., the proof of \autoref{thm:lifting}):
\begin{equation*}
\xymatrix{
H^0(E, \sO_E((1-p^e)K_E)) \ar[r]^{\delta} & H^0(E, \sO_E) \\
H^0(Y,\sO_Y((1-p^e)(K_Y+E)))  \ar@{^(->}[d] \ar[r]^-{\beta} \ar[u]^{\gamma} & \ar[u] H^0(Y, \sO_Y)   \ar@{=}[d] \\
H^0(X,\sO_Y((1-p^e)K_X)) \ar[r]^-{\alpha} & H^0(X, \sO_X)
}.
\end{equation*}
Since $X$ is affine, it is enough to show  the surjectivity of $\alpha$ for every $e \gg 0$. This follows from the surjectivity of $\beta$ for $e \gg 0$.   By assumption \autoref{itm:A_1:max_ideal}, and Nakayama's lemma, for that it is enough to show that both $\gamma$ and $\delta$ are surjective. 

The surjectivity of $\delta$ (for every integer $e >0$) follows from assumption \autoref{itm:A_1:F_split}. Indeed, if $E$ is globally $F$-split, then $\sO_E((1-p^e)K_E) \to \sO_E$ is surjective, as it is the dual of the split injection $\sO_E \to F^e_* \sO_E$.

Hence, we are left to show the surjectivity of $\gamma$ for $e \gg 0$.  By assumption \autoref{itm:A_1:F_split},  $E$ is Fano and hence $K_E$ is anti-ample. In particular, by \autoref{thm:S_0_equals_H_0}, $H^0(E, \sO_E((1-p^e)K_E)) = S^0(E ; \sO_E((1-p^e)K_E))$ for every integer $e \gg 0$. Hence, by \autoref{thm:lifting}, $S^0(Y,E; \sO_Y((1-p^e)(K_Y+E)) ) \subseteq H^0(Y,  \sO_Y((1-p^e)(K_Y+E)))$ surjects onto $H^0(E, \sO_E((1-p^e)K_E))$, which concludes our proof. 
\end{proof}

\section{Applications to higher dimensional algebraic geometry}
\label{sec:application}

The methods for finding sections presented in \autoref{sec:newer_methods} led recently to many advances in positive characteristic algebraic geometry. Some of them used the global generation result on Cartier modules presented in \autoref{sec:positivity_Cartier_modules} or the idea of its proof, and others used the lifting method presented in \autoref{sec:Frobenius_stable} and \autoref{sec:bounding_Frobenius_stable}. Even others were not using directly the tools of \autoref{sec:newer_methods}, but were building on other results proven by those methods. Below we summarize these results. Here we only list them, possibly indicating the method of their proof, and we refer to the original articles for the proofs or the more detailed statements. 

\subsection{Minimal Model Program}
\label{sec:MMP}

Probably the most prominent application is the, by now fairly complete, Minimal Model Program for $3$-folds. The main missing piece is abundance for Kodaira dimension $0$, although also some of the standard applications are also missing such as  most of the usual boundedness results. 

The story started before the dawn of the methods in \autoref{sec:newer_methods}, when Keel proved a base-point freeness theorem in the general type case \cite{Keel_Basepoint_freeness_for_nef_and_big_line_bundle_in_positive_characteristics} with unrelated methods. In his paper he also presented a Cone theorem for effective pairs, which was a folklore statement at the time. 

After Keel's results the main question was if flips exist for $3$-folds. This was shown in \cite{Hacon_Xu_On_the_three_dimensional_minimal_model_program_in_positive_characteristic} in the case of standard coefficients $\{n-1/n | n \in \bZ^{>0} \}$ for $p>5$, using mostly the lifting techniques of \autoref{sec:Frobenius_stable} and \autoref{sec:bounding_Frobenius_stable}, as well as the results about Seshadri constant discussed in \autoref{sec:Seshadri_constant}. These technical tools were used in the framework originally developed by Shokurov of proving existence of flips (c.f., \cite{Shokurov_Prelimiting_flips,Fujino_Special_termination_and_reduction_to_pl_flips}). An important ingredient is a generalization of Hara's result that Kawamata log terminal surface singularities are strongly $F$-regular if $p>5$ \cite{Hara_Classification_of_two_dimensional_F_regular_and_F_pure_singularities}, c.f., \autoref{cor:A_1}. In particular, this is the main reason for the $p>5$ restriction.

After the proof of existence of flips, most of the other papers were not using Frobenius techniques directly, but rater the techniques of the Minimal Model Program, building on the results of 
\cite{Hacon_Xu_On_the_three_dimensional_minimal_model_program_in_positive_characteristic}. An exception was \cite{Cascini_Tanaka_Xu_On_base_point_freeness_in_positive_characteristic}, where the Cascini, Tanaka and Xu proved global generation results using some very intricate version of the Frobenius techniques.  

Let us conclude \autoref{sec:MMP} by listing the above mentioned  papers that build on the results of \cite{Hacon_Xu_On_the_three_dimensional_minimal_model_program_in_positive_characteristic} rather than using the Frobenius techniques directly. In particular, these papers necessarily pertain to $3$-folds and $p>5$. First, Birkar showed using a clever MMP trick the existence of flips for arbitrary coefficients
\cite{Birkar_Existence_of_flips_and_minimal_models_for_3_folds_in_char_p}. In the same article he showed many other MMP related theorems, e.g., ACC for log canonical threshold, base point freeness, etc. Xu also gave an independent proof of the latter \cite{Xu_On_base_point_free_theorem_of_threefolds_in_positive_characteristic}. The already mentioned articles gave a full treatment of the general type case. The rest was wrapped up in the papers of Birkar and Waldron \cite{Birkar_Waldron_Existence_of_Mori_fibre_spaces_for_3_folds_in_char_p,Waldron_Finite_generation_of_the_log_canonical_ring_for_3-folds_in_char_p,Waldron_The_LMMP_for_log_canonical_3-folds_in_char_p}.

\subsection{Seshadri constants}
\label{sec:Seshadri_constant}

Musta{\c{t}}{\u{a}} and Schwede introduced a Frobenius versions $\epsilon_F(\sL,x)$ of Seshadri constants in \cite{Mustata_Schwede_A_Frobenius_variant_of_Seshadri_constants}, where $\sL$ is a line bundle, and $x \in X$ is smooth point. They could prove that if $\epsilon_F(\sL,x)>1$, then $\omega_X \otimes \sL$ is globally generated at $x$, and if $\epsilon_F(\sL,x)>2$, then $\sL$ defines a birational map. Furthermore, since $\epsilon(\sL,x)/\dim X  \leq \epsilon_F(\sL,x) \leq \epsilon(\sL,x)$, where $\epsilon(\sL,x)$ is the usual Seshadri constant, this implies the positive characteristic versions of the usual global generation results in characteristic zero given by Seshadri constants \cite[Prop 5.1.19]{Lazarsfeld_Positivity_in_algebraic_geometry_I}. The method of the proof is related to the technical tools presented in \autoref{sec:newer_methods}, although different in terms of details.

We note that similar global generation statements were shown for surfaces in  \cite{Cerbo_Fanelli_Effective_Matsusaka_s_Theorem_for_surfaces_in_characteristic_p} by methods not discussed in the present article. There is no assumption here of Seshadri type, so the result says that if $L$ is an ample Cartier divisor on a smooth surface $X$, then $2K_X + 38 L$ is very ample. These are the smallest  constants known so far for surfaces for Fujita type very ampleness statements.

\subsection{Semi- and weak-positivity}

In characteristic zero, statements proving semi-positivity (i.e., nefness) or weak-positivity (roughly the version of pseudo-effectivity for vector bundles) of $f_* \omega_{X/T}$ for fiber spaces $f : X \to T$ have been abundant and central \cite{Griffiths_Periods_of_integrals,Fujita_On_Kahler_fiber_spaces,Kawamata_Characterization_of_abelian_varieties,Viehweg_Weak_positivity,Kollar_Subadditivity_of_the_Kodaira_dimension}. In positive characteristic, similar results were obtained only recently by the author of the article and Ejiri \cite{Patakfalvi_Semi_positivity_in_positive_characteristics,Patakfalvi_On_subadditivity_of_Kodaira_dimension_in_positive_characteristic_over_a_general_type_base,Ejiri_Weak_positivity_theorem_and_Frobenius_stable_canonical_rings_of_geometric_generic_fibers,Ejiri_Positivity_of_anti-canonical_divisors_and_F-purity_of_fibers} via (all) the methods of \autoref{sec:newer_methods}, plus a careful study of the behavior of all this in families in 
\cite{Patakfalvi_Schwede_Zhang_F_singularities_in_families}. Similar methods yield subadditivity of Kodaira dimension type statements, that is, that $\kappa(X) \geq \kappa(T) + \kappa(F)$ for a fiber space $f : X \to T$ with geometric general fiber $F$ \cite{Patakfalvi_On_subadditivity_of_Kodaira_dimension_in_positive_characteristic_over_a_general_type_base,Ejiri_Iitaka_s_C_n_m_conjecture_for_3-folds_in_positive_characteristic, Zhang_Subadditivity_of_Kodaira_dimensions_for_fibrations_of_three-folds_in___positive_characteristics}. More precisely, the author of the present article  showed  in \cite{Patakfalvi_On_subadditivity_of_Kodaira_dimension_in_positive_characteristic_over_a_general_type_base} subadditivity of Kodaira dimension if the base is of general type and the Hasse-Witt matrix (c.f., \autoref{cor:pushforward}) of the geometric generic fiber is not nilpotent.  In the other direction, when the dimension is fixed, but the other assumptions are relaxed, Ejiri  showed full subadditivity of Kodaira 
dimension for $3$-folds in \cite{Ejiri_Iitaka_s_C_n_m_conjecture_for_3-folds_in_positive_characteristic}. This result was earlier obtain over finite fields, and their algebraic closure by Birkar, Chen and Zheng \cite{Birkar_Chen_Zhang_Iitaka_s_C_n_m_conjecture_for_3-folds_over_finite_fields}. However, instead of using Frobenius techniques directly, they used \cite{Patakfalvi_Semi_positivity_in_positive_characteristics} together with the MMP results explained in \autoref{sec:MMP}. Furthermore, 
\cite{Chen_Zhang_The_subadditivity_of_the_Kodaira-dimension_for_fibrations_of_relative_dimension_one_in_positive_characteristics} shows the relative $1$ dimensional case of subadditivity of Kodaira dimension by completely unrelated techniques originating from \cite{Viehweg_Canonical_divisors_and_the_additivity_of_the_Kodaira_dimension_for_morphisms_of_relative_dimension_one}. 

\subsection{Abelian varieties, generic vanishing and varieties of maximal Albanese dimension}

There has been much known classically about abelian varieties themselves in positive characteristic (e.g., \cite{Mumford_Abelian_varieties}), and not much has been added to that using Frobenius techniques. On the other hand, there have been many interesting new results concerning divisors or cohomology on abelian varieties, or varieties mapping to abelian varieties. First,  Hacon showed in \cite{Hacon_Singularities_of_pluri_theta_divisors_in_Char_p} a multiplicity bound on the linear systems of Theta-divisors using techniques as in \autoref{sec:newer_methods}. Next, Zhang obtained with similar methods that $|4 K_X|$ is birational for varieties of maximal Albanese dimension with separable Albanese morphism 
\cite{Zhang_Pluri-canonical_maps_of_varieties_of_maximal_Albanese_dimension_in___positive_characteristic}. 

Later, Hacon and the author of the present article showed a   generic vanishing type theorem for Cartier modules \cite{Hacon_Patakfalvi_A_generic_vanishing_in_positive_characteristic}, which (weakly) mirrors the generic vanishing theorems available for holonomic $\sD$-modules over $\bC$ \cite{Schnell_Holonomic_D-modules_on_abelian_varieties}. This then applied to the Cartier module $a_* \omega_X$ (\autoref{ex:Cartier_module_dualizing}), where $a: X \to A$ is the Albanese morphism of a projective variety, yields generic vanishing results similar to the classical ones in characteristic zero. An example statement is that for a Cartier module $\sM$ on an abelian variety $A$ and for any integer $i>0$, the natural homomorphism $H^i(A, \sL \otimes F_*^e \sM) \to H^i(A, \sL \otimes \sM)$ is zero for every very generic $\sL \in \Pic^0(A)$ and integer $e \gg 0$ (where the bound on $e$ depends on $\sL$) \cite[Cor 3.3.1]{Hacon_Patakfalvi_A_generic_vanishing_in_positive_characteristic}. Later some of the statements have 
been improved in \cite{Watson_Zhang_On_the_generic_vanishing_theorem_of_Cartier_modules} by Watson and Zhang.

One consequence of the above generic vanishing theorem is a characterization of ordinary abelian varieties. An abelian variety $A$ is ordinary if the action of the Frobenius on $H^{\dim A}(A,\sO_X)$ is bijective, or equivalently $S^0(A, \sO_A)\neq 0$, using \autoref{def:S_0}. Then in \cite[Thm 1.1.1.a]{Hacon_Patakfalvi_A_generic_vanishing_in_positive_characteristic} it is shown that a smooth projective variety $X$ is birational to an ordinary abelian variety if and only if its first Betti number is $2 \dim X$,  $\kappa_S(X)=0$ (where $\kappa_S$ is the version of Kodaira dimension defined using $S^0(\_)$ instead of $H^0(\_)$ 
\cite[4.1]{Hacon_Patakfalvi_A_generic_vanishing_in_positive_characteristic}), and the degree of the Albanese map of $X$ is prime-to-$p$. Similar methods yielded also a characterization of abelian varieties in \cite{Hacon_Patakfalvi_On_charactarization_of_abelian_varieties_in_characteristic_p} (so no ordinarity): $X$ is birational to an abelian variety if and only if $\kappa(X)=0$ and the Albanese map is generically finite of degree prime-to-$p$ over its image. Furthermore, using the results of \cite{Hacon_Patakfalvi_A_generic_vanishing_in_positive_characteristic}, Sannai and Tanaka proved another characterization of ordinary ablian varieties \cite{Sannai_Tanaka_A_characterization_of_ordinary_abelian_varieties_by_the_Frobenius___push-forward_of_the_structure_sheaf}: $X$ is an ordinary abelian variety if and only if $K_X$ is pseudo-effective and $F^e_* \sO_X$ is a direct sum of invertible sheaves for infinitely many integers $e>0$.  

In a different direction, Wang showed generic vanishing type statement for surfaces lifting to $W_2(k)$ \cite{Wang_Generic_vanishing_and_classification_of_irregular_surfaces_in_positive___characteristics},  which he used for classification statements for surfaces. 

\subsection{Numerical dimensions}

In \cite{Mustata_The_non_nef_locus_in_positive_characteristic}, Musta{\c{t}}{\u{a}} showed the characterization of the non-nef locus, also called diminished base-locus,  using the asymptotic order of vanishing, which has been known for a while in characteristic zero \cite{Ein_Lazarsfeld_Musta_Nakamaye_Popa_Asymptotic_invariants_of_base_loci}. The method is a generalization of that of \autoref{thm:Cartier_module_globally_generated}. Related methods are used in \cite{Cascini_Hacon_Mustata_Schwede_On_the_numerical_dimension_of_pseudo_effective_divisors_in_positive_characteristic} to show that if a pseudo-effective divisor had numerical dimension $0$, then it is numerically equivalent to the negative part of its divisorial decomposition (see \cite[Thm V.1.12]{Nakayama_Noboru_Zariski_decomposition_and_abundance} for the characteristic zero version). 

We also mention that by methods not discussed in this article (see \cite{Keel_Basepoint_freeness_for_nef_and_big_line_bundle_in_positive_characteristics}), Cascini, {M\raise .575ex \hbox{\text{c}}Kernan} and Musta{\c{t}}{\u{a}} showed that the augmented base-locus of a nef line bundle is equal to its exceptional set, that is the union of those irreducible subvarieties over which it has zero self-intersection
\cite{Cascini_McKernan_Mustata_The_augmented_base_locus_in_positive_characteristic} (see \cite{Nakamaye_Stable_base_loci_of_linear_series} for the characteristic zero version).

\subsection{Rationally connectedness and Fano varieties}

In the past few years there have been quite a few results about rationally chain connectedness of $3$-folds. First, in  \cite{Gongyo_Li_Patakfalvi_Schwede_Tanaka_Zong_On_rational_connectedness_of_globally_F_regular_threefolds} it was shown that globally $F$-regular  $3$-folds are rationally chain connected if $p \geq 11$. The method uses the MMP for $3$-folds (\autoref{sec:MMP}) together with some facts about globally $F$-regular varieties.  Later,  this result was generalized to Fano type $3$-folds for $p >5$ in \cite{Gongyo_Nakamura_Tanaka_Rational_points_on_log_Fano_threefolds_over_a_finite_field} by Gongyo, Nakamura and Tanaka. The latter article used MMP, particularly MMP for Fano type varieties, which was not available to the fullest generality at the time of 
\cite{Gongyo_Li_Patakfalvi_Schwede_Tanaka_Zong_On_rational_connectedness_of_globally_F_regular_threefolds}. Furthermore, Wang also showed a relative version of these results in 
\cite{Wang_On_relative_rational_chain_connectedness_of_threefolds_with_anti-big___canonical_divisors_in_positive_characteristics}.

We note that using the above rationally chain connected property, it was also shown in  \cite{Gongyo_Nakamura_Tanaka_Rational_points_on_log_Fano_threefolds_over_a_finite_field} that Fano type varieties over $\bF_q$ have a rational point. In fact, it was shown that they have $W \sO$-rational singularities, from which the above (and also a more precise) rational point statement follows using \cite{Blickle_Esnault_Rational_singularities_and_rational_points}. This generalizes to the singular setting the $3$-fold case of  Esnault's result on rational points of Fano manifolds \cite{Esnault_Varieties_over_a_finite_field_with_trivial_Chow_group_of_0_cycles_have_a_rational_point}.

Also, about anti-canonically polarized varieties, Ejiri proved in \cite{Ejiri_Positivity_of_anti-canonical_divisors_and_F-purity_of_fibers} that if $f: X \to Y$ is a flat fiber space (including smoothness of $X$ and $Y$) with strongly $F$-regular fibers such that $-K_X$ is nef and big, then so is $-K_Y$. This is  a generalization of the consequence of the Hurwitz formula stating that the finite image of a rational curve is a rational curve. Different versions of the statement have been known for a while in characteristic $0$ (see \cite{Ejiri_Positivity_of_anti-canonical_divisors_and_F-purity_of_fibers} for references).

\subsection{Canonical bundle formula}

Canonical bundle formulas generalize the classical formula of Kodaira for elliptic fibrations, relating the (log-)canonical divisor of the base of a fibration the fibers of which have trivial (log-)canonical divisors to the (log-)canonical divisor of its total space. First, Das and Schwede showed such a formula for pairs with indices prime-to-$p$ and globally $F$-split generic fibers \cite{Das_Schwede_The_F-different_and_a_canonical_bundle_formula}. This had consequences on adjunction theory of high codimension log canonical centers. It was shown in the same article that for such a center the usual different is always smaller than the $F$-different. In fact, in codimension $1$ the two differents agree, a statement showed earlier also by Das \cite[Thm 5.3]{Das_On_strongly_F_regular_inversion_of_adjunction}.
In the same direction, also using a canonical bundle formula, not requiring $F$-singularity assumption but only working for families of curves, Hacon and Das drew consequences on adjunction theory of one dimensional log canonical centers on $3$-folds \cite{Das_Hacon_On_the_Adjunction_Formula_for_3-folds_in_characteristic_p>5}.

\subsection{Kodaira type vanishings, surjectivity statements and liftability to characteristic $0$}

Using methods related to \autoref{thm:S_0_equals_H_0}, Tanaka showed that on a smooth projective surface, although Kodaira vanishing fails, $H^1(K_X + A + mN)=0$, where $A$ is an ample divisor, $N$ is a nef, numerically non-trivial divisor and $m \gg0$ is a big enough integer \cite{Tanaka_The_X-method_for_klt_surfaces_in_positive_characteristic}. In the same article, he drew consequences for the log-MMP on surfaces. 

A usual use of Kodaira vanishing in characteristic $0$ birational geometry is to show lifting statements for pluricanonical forms as follows: Assume that $K_X +S +A $ is nef, where for simplicity assume that $X$ is smooth and $A$ and $S$ are divisors on $X$. Assume also that $A$ is ample, $S$  is a smooth and effective and $K_X +S +A$ is nef. Then Kodaira vanishing implies that $H^1(X, \sO_X(K_X + A + (m-1) (K_X + S +A )))=0$ for any integer $m>0$. In particular, this implies that $H^0(X, \sO_X(m(K_X + S +A ))) \to H^0(S, \sO_S(m(K_S + A|_S)))$ is surjective. Tanaka showed in 
\cite{Tanaka_The_trace_map_of_Frobenius_and_extending_sections_for_threefolds} that the same surjectivity holds for positive characteristic $3$-folds assuming that $\kappa(K_S + A|_S) \neq 0$ and $m \gg 0$. 

It has been known for a while that despite the failure of Kodaira vanishing for surfaces in general \cite{Raynaud_Contre_exemple_au_vanishing_theorem}, it does hold for special surfaces that are not elliptic of Kodaira dimension $1$ \cite{Ekedahl_Canonical_models_of_surfaces_of_general_type_in_positive_characteristic,Shepherd_Barron_Unstable_vector_bundles_and_linear_systems_on_surfaces_in_characteristic_p}. However, there has been not much known about the Kawamata-Viehweg vanishing theorem, as its usual proof involves passage to certain branch covers that are typically of general type. Cascini, Tanaka and Witaszek showed recently that Kawamta-Viehweg vanishing does hold for surfaces of Fano type for $p \geq p_0$ \cite{Cascini_Tanaka_Witaszek_On_log_del_Pezzo_surfaces_in_large_characteristic}, where $p_0$ is some integer. Unfortunately, only the existence of such a $p_0$ is known so far, although knowing it explicitly would be very useful. The above Kawamata-Viehweg vanishing is deduced from the other result 
of \cite{Cascini_Tanaka_Witaszek_On_log_del_Pezzo_surfaces_in_large_characteristic}, which states that over a certain prime a Fano-type  variety $X$ together with the boundary $B$ guaranteeing the Fano type condition is either  globally $F$-regular or it has a log-resolution lifting to characteristic zero (here the bound on the prime depends on the coefficients of $B$).  In a related direction Zdanowicz showed different connections between $F$-singularities and liftability to $W_2(k)$ together with the Frobenius
\cite{Zdanowicz_Liftability_of_singularities_and_their_Frobenius_morphism_modulo_p_2}.

\bibliographystyle{skalpha}
\bibliography{includeNice}

\end{document}